\documentclass[a4paper,french,twoside,draft,english,11pt,centertags]{smfartVS}

\usepackage{smfenum,smfthm,amsmath,amssymb,mathrsfs,euscript,color}
\usepackage{stmaryrd,mathabx}
\usepackage[latin1]{inputenc}
\input{xypic}

\numberwithin{equation}{section} 
\theoremstyle{plain}

\newcounter{nonumber}

\newtheorem{theon}[nonumber]{Theorem}

\def\FF{\mathbf{F}}
\def\NN{\mathbf{N}}

\def\ZZ{\mathbf{Z}} 

\def\A{{\rm A}}
\def\B{{\rm B}}

\def\D{{\rm D}}
\def\E{{\rm E}}
\def\F{{\rm F}}
\def\G{{\rm G}}
\def\H{{\rm H}}
\def\I{{\rm I}}
\def\J{{\rm J}}
\def\K{{\rm K}}
\def\L{{\rm L}}
\def\M{{\rm M}}
\def\N{{\rm N}}
\def\O{{\rm O}}
\def\P{{\rm P}}
\def\Q{{\rm Q}}
\def\R{{\rm R}}
\def\SS{{\rm S}}

\def\U{{\rm U}}
\def\V{{\rm V}}
\def\W{{\rm W}}
\def\X{{\rm X}}
\def\Y{{\rm Y}}
\def\Z{{\rm Z}}

\def\Aa{\mathscr{A}}

\def\Cc{\mathscr{C}}

\def\Gg{\mathscr{G}}
\def\Hh{\mathscr{H}}

\def\Ll{\mathscr{L}}
\def\Mm{\mathscr{M}}
\def\Nn{\mathscr{N}}
\def\Oo{\mathscr{O}}
\def\Pp{\mathscr{P}}

\def\Rr{\mathscr{R}}

\def\AA{\mathfrak{A}}

\def\PP{\mathfrak{P}}
\def\TT{\boldsymbol{\Theta}}

\def\a{\alpha} 
\def\b{\beta}

\def\g{\gamma}
\def\h{\varphi}
\def\k{\kappa}
\def\l{\lambda}
\def\n{\eta}

\def\p{\mathfrak{p}}
\def\s{\sigma}
\def\t{\theta}

\def\w{\varpi}

\def\Ga{\Gamma}
\def\La{\Lambda}
\def\Om{\Omega}
\def\om{\omega}

\def\>{\geqslant}
\def\<{\leqslant}
\def\odo{\otimes\dots\otimes}
\def\tdt{\times\dots\times}

\def\Hom{{\rm Hom}}
\def\End{{\rm End}}

\def\Mat{{\rm M}}
\def\GL{{\rm GL}}
\def\Gal{{\rm Gal}}
\def\Ker{{\rm Ker}}

\def\tr{{\rm tr}}

\def\id{{\rm id}}

\def\Ind{{\rm Ind}}
\def\ind{{\rm ind}}

\def\mult#1{{#1}^{\times}}

\def\Cc{\EuScript{C}}
\def\Hh{\EuScript{H}}
\def\Oo{\EuScript{O}}
\def\fb{\bar{f}}

\def\kk{\mathfrak{k}}

\def\nn{\mathfrak{n}}

\def\tt{\mathfrak{t}}
\def\aa{\mathfrak{a}}

\def\ll{l}

\def\CR{{\rm R}}

\def\Irr{{\rm Irr}}

\def\cusp{{\rm cusp}}
\def\scusp{{\rm scusp}}

\def\ia{\Ind}
\def\ip{\boldsymbol{i}}
\def\St{{\rm St}}

\def\kmax{\k_{{\rm max}}}
\def\tmax{\t_{{\rm max}}}
\def\jmax{\J_{{\rm max}}}

\def\Lamax{\La_{{\rm max}}}

\def\bkmax{\bk_{{\rm max}}}

\def\bjmax{\BJ_{{\rm max}}}

\def\BJ{{\bf J}}
\def\BH{{\bf H}}
\def\bl{\boldsymbol{\l}}
\def\bk{\boldsymbol{\k}}
\def\bs{\boldsymbol{\s}}
\def\bn{\boldsymbol{\n}}
\def\bt{\boldsymbol{\t}}
\def\KM{\textbf{\textsf{K}}}
\def\KB{\textbf{\textsf{K}}_\SS}
\def\TM{\textbf{\textsf{T}}} 

\def\vr{\varrho}

\def\({\left(}
\def\){\right)}

\def\ffr#1{\smash{\mathop{\ \longrightarrow\ }\limits^{#1}}}

\def\sr{\mathsf{s}}
\def\Ext{{\rm Ext}}

\def\labelenumi{{\rm(\roman{enumi})}}
\def\theenumi{\roman{enumi}}
\def\labelenumii{{\rm(\alph{enumii})}}
\def\theenumii{\alph{enumii}}

\author{Vincent Sécherre}
\address{Université de Versailles Saint-Quentin-en-Yvelines\\
Laboratoire de Mathémati\-ques de Versailles\\
45 avenue des Etats-Unis\\
78035 Versailles cedex, France}
\email{vincent.secherre@math.uvsq.fr}

\author{Shaun Stevens}
\address{School of Mathematics, University of East Anglia, 
  Norwich NR4 7TJ, United Kingdom}
\email{Shaun.Stevens@uea.ac.uk}

\title[Blocks for $\ell$-modular smooth representations of $\GL_{m}(\D)$]
{Block decomposition of the category of $\ell$-modular smooth 
representations of $\GL_{n}(\F)$ and its inner forms}

\begin{altabstract}
Let $\F$ be a nonarchimedean locally compact field of residue 
characteristic $p$, let $\D$ be a finite dimensional central 
division $\F$-algebra and let $\R$ be an algebraically closed field 
of cha\-rac\-teristic different from $p$. 
To any irreducible smooth representation of $\G=\GL_m(\D)$, 
$m\>1$ with coef\-ficients in $\R$, we can attach a uniquely determined 
inertial class of supercuspidal pairs of $\G$. 
This provides us with a partition of the set of all isomorphism classes 
of irreducible representations of $\G$. 
We write $\Rr(\G)$ for the category of all smooth representations of 
$\G$ with coefficients in $\R$. 
To any inertial class $\Om$ of supercuspidal pairs of $\G$, we can attach the 
subcategory $\Rr(\Om)$ made of smooth representations all of whose 
irreducible subquotients are in the subset determined by this inertial class. 
We prove that the category $\Rr(\G)$ 
decomposes into the product of the $\Rr(\Om)$'s, where $\Om$ ranges over all 
possible inertial class of supercuspidal pairs of $\G$, 
and that each summand $\Rr(\Om)$ is indecomposable. 
\end{altabstract}

\long\def\MSC#1\EndMSC{\def\arg{#1}\ifx\arg\empty\relax\else
     {\par\narrower\noindent%
     2010 Mathematics Subject Classification: #1\par}\fi}

\long\def\KEY#1\EndKEY{\def\arg{#1}\ifx\arg\empty\relax\else
	{\par\narrower\noindent Keywords and Phrases: #1\par}\fi}

\begin{document}

\maketitle

\MSC 
22E50 
\EndMSC

\KEY 
Modular representations of $p$-adic reductive groups, 
Semisimple types, Inertial classes, Supercuspidal support,
Blocks
\EndKEY

% {\hfill\textcolor{green}{\bf\today}} 

% \tableofcontents

\sectionNN*{Introduction}

When considering a category of representations of some group or algebra, a 
natural step is to attempt to decompose the category into \emph{blocks}; that 
is, into subcategories which are indecomposable summands.  Thus any 
representation can be decomposed uniquely as a direct sum of pie\-ces, one in 
each block; any morphism comes as a product of morphisms, one in each block; 
and this decomposition of the category is the finest decomposition for which 
these properties are satisfied. Then a full understanding of the category is 
equivalent to a full understanding of all of its blocks.  

In the case of representations of a finite group~$\G$, over an algebraically 
closed field~$\R$, there is always a block decomposition. In the simplest 
case, when the characteristic of~$\R$ is prime to the order of~$\G$, this is 
particularly straightforward: all representations are semisimple so each block 
consists of representations isomorphic to a direct sum of copies of a fixed 
irreducible representation.  
In the general case, there is a well-developed 
theory, beginning with the work of Brauer and Nesbitt, and understanding the block 
structure is a major endeavour.  

Now suppose~$\G$ is the group of rational points of a connected reductive algebraic 
group over a nonarchimedean locally compact field $\F$, of residue 
characteristic~$p$. When~$\R$ has characteristic zero, a block decomposition 
of the category~$\Rr_\R(\G)$ of smooth~$\R$-representations of~$\G$ was given 
by Bernstein~\cite{BDKV}, in terms of the classification of representations 
of~$\G$ by their cuspidal support. Any irreducible representations~$\pi$ 
of~$\G$ is a quotient of some (normalized) parabolically induced 
representation~$\ip_{\M}^\G\vr$, with~$\vr$ a  
cuspidal irreducible representation of a 
Levi subgroup~$\M$ of~$\G$; the pair~$(\M,\vr)$ is determined up 
to~$\G$-conjugacy by~$\pi$ and is called its \emph{cuspidal support}. Then 
each such pair~$(\M,\vr)$ determines a block, whose objects are those 
representations of~$\G$ all of whose subquotients have cuspidal 
support~$(\M,\vr\chi)$, for some unramified character~$\chi$ of~$\M$.  

One important tool in proving this block decomposition is the equivalence of 
the following two properties of an irreducible~$\R$-representation~$\pi$ 
of~$\G$: 
\begin{itemize}
\item[$\bullet$] 
$\pi$ is not a quotient of any properly parabolically induced 
representation; equivalently, all proper Jacquet modules of~$\pi$ 
are zero ($\pi$ is \emph{cuspidal}); 
\item[$\bullet$] 
$\pi$ is not a \emph{sub}quotient of any properly 
parabolically induced representation~$\ip_{\M}^\G\vr$
with~$\vr$ an ir\-re\-ducible representation ($\pi$ is \emph{supercuspidal}). 
\end{itemize}
When~$\R$ is an algebraically closed field of positive characteristic different 
from~$p$ (the \emph{modular} case), %~$\ell\ne p$, 
these properties are no longer equivalent and the methods used in the 
characteristic zero case cannot be applied.  
Instead, one can attempt to define the \emph{supercuspidal support} of a 
smooth irreducible~$\R$-representation~$\pi$ of~$\G$: it is a pair~$(\M,\vr)$ 
consisting of an irreducible supercuspidal representation~$\vr$ of a Levi subgroup~$\M$ 
of~$\G$ such that~$\pi$ is a \emph{subquotient} 
of~$\ip_{\M}^\G\vr$. However, for a general group~$\G$, it is not known 
whether the supercuspidal support of a representation is well-defined up to 
conjugacy; indeed, the analogous question for finite reductive groups of Lie 
type is also open. 

In any case, one can define the notion of an 
\emph{inertial supercuspidal class}~$\Omega=[\M,\vr]_\G$: 
it is the set of pairs~$(\M',\vr')$, consisting
of a Levi subgroup~$\M'$ of~$\G$ and a supercuspidal represen\-ta\-tion~$\vr'$
of~$\M'$, which are~$\G$-conjugate to~$(\M,\vr\chi)$, for some unramified
character~$\chi$ of~$\M$.  Given such a class~$\Omega$, we denote
by~$\Rr_\R(\Omega)$ the full subcategory of~$\Rr_\R(\G)$ whose objects are
those representations all of whose subquotients are isomorphic to a 
subquotient of~$\ip_{\M'}^\G\vr'$, for some~$(\M',\vr')\in\Omega$.  

The main purpose of this paper is then to prove the following result:
\begin{theon} 
Let~$\G$ be an inner form of~$\GL_n(\F)$ and let~$\R$ be an 
algebraically closed field of characteristic different from~$p$. 
Then there is a block decomposition 
\[
\Rr_\R(\G)=\prod_\Omega \Rr_\R(\Omega),
\]
where the product is taken over all inertial supercuspidal classes.
\end{theon}

This theorem generalizes the Bernstein decomposition in the case that~$\R$ has
characteristic zero, and also a similar statement, for general~$\R$, stated by 
Vign\'eras~\cite{Vselecta} in the split ca\-se~$\G=\GL_n(\F)$; however, the 
authors were unable to follow all the steps in~\cite{Vselecta} so our proof is 
independent, even if some of the ideas come from there. 
% (notably Proposition~\ref{AdamTrask}).  

Our proof builds on work of M{\'{\i}}nguez and the first
author~\cite{MS1,MS2}, in which they give a clas\-si\-fication of the 
irreducible~$\R$-representations of~$\G$, in terms of supercuspidal 
representations, and of the supercuspidal representations in terms of the 
theory of types.  
In particular, they prove that supercuspidal support is
well-defined up to conjugacy, so that the irreducible objects
in~$\Rr_\R(\Omega)$ are precisely those with supercuspidal support
in~$\Omega$. 

One question we do not address here is the structure of the
blocks~$\Rr_\R(\Omega)$. 
Given the explicit results on supertypes here, it is 
not hard to construct a progenerator~$\Pi$ for~$\Rr_\R(\Omega)$ as a 
compactly-induced representation; for~$\G=\GL_n(\F)$ this was done 
(independently) by Guiraud~\cite{G} (for level zero blocks)
and Helm~\cite{H}. Then~$\Rr_\R(\Omega)$ 
is equivalent to the category of~$\End_\G(\Pi)$-modules. In the case that~$\R$ 
has characteristic zero, the algebra~$\End_\G(\Pi)$ was described as a tensor 
product of affine Hecke algebras of type~$\A$ in~\cite{SeSt2} 
(or~\cite{BKsemi} in the split case); indeed, we use this description in our 
proof here.  
For~$\R$ an algebraic closure $\overline\FF_\ell$ of a finite field of 
characteristic~$\ell\ne p$, and a 
block~$\Rr_\R(\Omega)$ with~$\Omega=[\GL_n(\F),\vr]_{\GL_n(\F)}$, 
Dat~\cite{DatLTNAE} has 
described this algebra; it is an algebra of Laurent polynomials in one 
variable over the~$\R$-group algebra of a cyclic~$\ell$-group.
It would be interesting to obtain a description in the general case.  

We now  describe the proof of  the theorem, which relies  substantially on the
theory  of semisimple types  developed in~\cite{SeSt2}  (see~\cite{BKsemi} for
the split case). Given  an inner form~$\G$ of~$\GL_n(\F)$, in~\cite{SeSt2} the
authors  constructed a family  of pairs~$(\BJ,\bl)$,  consisting of  a compact
open  subgroup~$\BJ$ of~$\G$ and  an irreducible  complex representation~$\bl$
of~$\BJ$. This family of pairs~$(\BJ,\bl)$, called semisimple types, satisfies
the following condition:  for every inertial cuspidal class~$\Om$,  there is a
semisimple type~$(\BJ,\bl)$ such  that the irreducible complex representations
of~$\G$  with cuspidal support  in~$\Om$ are  exactly those  whose restriction
to~$\BJ$ contains~$\bl$. 

In~\cite{MS1}, M{\'{\i}}nguez and the  first author extended this construction
to the modular case: they con\-struc\-ted a family of pairs~$(\BJ,\bl)$,
consisting of a compact open subgroup~$\BJ$ of~$\G$ and an ir\-re\-du\-cible complex
representation~$\bl$ of~$\BJ$, called semisimple supertypes. However, they did
not  give  the  relation  between  these semisimple  supertypes  and  inertial
supercuspidal classes of~$\G$. In this paper, we prove: 
\begin{itemize}
\item 
for each inertial supercuspidal class~$\Om$, there is a semisimple 
supertype~$(\BJ,\bl)$ such that the irreducible~$\R$-representations of~$\G$ 
with supercuspidal support in~$\Om$ are precisely those which appear as 
subquotients of the compactly induced representation~$\ind^\G_{\BJ}(\bl)$; 
\item 
two semisimple supertypes~$(\BJ,\bl)$ and~$(\BJ',\bl')$ correspond to 
the same inertial supercuspidal class if and 
only if the compactly induced representations $\ind^\G_{\BJ}(\bl)$ and 
$\ind^\G_{\BJ'}(\bl')$ are isomorphic, 
in which case we say the supertypes are equivalent. 
\end{itemize}
Thus we get a bijective correspondence between the inertial supercuspidal 
classes for~$\G$ and the equivalence classes of semisimple supertypes. 

To each semisimple supertype, we attach a crucial tool, already used 
in~\cite{MS2} to obtain the clas\-si\-fication of the 
irreducible~$\R$-representations of~$\G$. 
This is a functor which associates, 
to each smooth~$\R$-representation of~$\G$, a representation of the finite 
reductive quotient of~$\BJ$.  
More precisely, given a semisimple supertype~$(\BJ,\bl)$, 
there is a normal compact open subgroup~$\BJ^1$ of~$\BJ$ such that:
\begin{itemize}
\item 
the quotient~$\BJ/\BJ^1$ is isomorphic to a group of the 
form~$\GL_{n_1}(\kk_1)\tdt\GL_{n_r}(\kk_r)$, where~$\kk_i$ is a finite 
extension of the residue field of~$\F$ and~$n_i$ is a positive integer, 
for~$i\in\{1,\dots,r\}$; 
\item 
the representation~$\bl$ decomposes (non-canonically) 
as~$\bk\otimes\bs$, where~$\bk$ is a particular irreducible representation 
of~$\BJ$ and~$\bs$ is the inflation to~$\BJ$ of a supercuspidal 
irreducible representation of~$\GL_{n_1}(\kk_1)\tdt\GL_{n_r}(\kk_r)$; 
\item
in the particular case where the semisimple supertype is 
homogeneous (see \S\ref{ss:hom}), 
there is a normal compact open subgroup~$\BH^1$ of~$\BJ^1$ such that the 
restriction of $\bk$ to $\BH^1$ is a direct sum of copies of a certain 
character~$\bt$, called a simple character. 
\end{itemize}
Given a choice of decomposition~$\bl=\bk\otimes\bs$, we define a functor
\begin{equation*}
\KM=\KM_{\bk}:\pi\mapsto\Hom_{\BJ^1}(\bk,\pi)
\end{equation*}
from~$\Rr(\G)$ to~$\Rr(\BJ/\BJ^1)$, with~$\BJ$ acting on~$\Hom_{\BJ^1}(\bk,\pi)$ via
\begin{equation*}
x\cdot f=\pi(x)\circ f\circ \bk(x)^{-1}, 
\end{equation*}
for all~$x\in\BJ$ and~$f\in\Hom_{\BJ^1}(\bk,\pi)$. Since~$\BJ^1$ is a pro-$p$ 
group, this functor is exact. 

An important property of this functor $\KM$ is its behaviour with respect to 
parabolic induction (see Theorem~\ref{Cagliostro}): for a parabolic subgroup 
of~$\G$ compatible with the data involved in the con\-struc\-tion 
of~$(\BJ,\bl)$, this functor commutes with parabolic induction.  
This result is related to a remarkable property of simple characters
(see Lemma~\ref{charN1}) 
which, to our knowledge, was not previously known even in the split case. 

This allows a somewhat surprising back-and-forth argument between the complex 
case, where the compatibility of~$\KM$ with parabolic induction was already 
known (see~\cite{SZ}), and the modular case; this is because, 
in the case of a homogeneous supertype, 
the condition on the simple character~$\bt$
holds for~$\R$-representations if and only if it holds for com\-plex 
representations, since~$\BH^1$ is a pro-$p$ group (see the proof of 
Proposition~\ref{MainTheoK}). 
This is the objective of sections~\ref{S1} 
to~\ref{bijomjl}, and requires the notion of endo-class developed 
in~\cite{SeSt1} (see~\cite{BH} in the split case). 

Now we need to define the subcategories of~$\Rr_\R(\G)$ which will be the 
blocks we seek, which we do in section~\ref{S7}.  To each semisimple 
supertype~$(\BJ,\bl)$ we associate a full subcategory~$\Rr_\R(\BJ,\bl)$, whose 
objects are those smooth representations~$\V$ which are generated by the 
maximal subspace of~$\KM(\V)$ all of whose irreducible subquotients have 
supercuspidal support in a fixed set determined by~$\bs$ (see 
Definition~\ref{DefELS}).  This subcategory is independent of the choice of 
decomposition~$\bl=\bk\otimes\bs$.  Note that the existence of a maximal 
subspace of~$\KM(\V)$ with the required property depends on a decomposition of 
the category of representations of the finite reductive group 
\begin{equation*}
\BJ/\BJ^1\simeq\GL_{n_1}(\kk_1)\tdt\GL_{n_r}(\kk_r)
\end{equation*} 
in terms of supercuspidal support (the unicity of which is one of the 
principal results of~\cite{James}).  Moreover, it follows from this 
decomposition that~$\Rr_\R(\G)$ decomposes as a product of the 
subcategories~$\Rr_\R(\BJ,\bl)$, where~$(\BJ,\bl)$ runs through the 
equivalence classes of semisimple supertypes. 

It  remains only to  prove that  the~$\Rr_\R(\BJ,\bl)$ are  indecomposable and
coincide with the~$\Rr_\R(\Om)$, which  is the purpose of section~\ref{S8}. To
prove the indecomposability of the~$\Rr_\R(\BJ,\bl)$ we use the endo\-morphism
algebra of the compactly induced representation~$\ind^\G_\BJ(\bl)$, whose
structure was determined in~\cite{SeSt2} 
(and \cite{MS1} for the modular case).
The centre of this algebra is an
integral domain, which implies that~$\ind^\G_\BJ(\bl)$ is
indecomposable.   Since  its  irreducible   subquotients  coincide   with  the
irreducible objects of~$\Rr_\R(\BJ,\bl)$, it  follows that this subcategory is
indecomposable. 

We end the paper, in section~\ref{S9}, by proving a remarkable property of 
supercuspidality: if an irreducible representation of~$\G$ does not appear as 
a subquotient any properly parabolically induced 
representation~$\ip_{\M}^\G\vr$, with~$\vr$ irreducible, then it also does not 
appear as a subquotient of~\emph{any} properly parabolically induced 
representation. 

%%%%%%%%%%%%%%%%%%%%%%%%%%%%%%%%%%%%%%%%%%%%%%%%%%%%%%%%%%%%%%%%%%%%%%%%
\sectionNN*{Notation}
%%%%%%%%%%%%%%%%%%%%%%%%%%%%%%%%%%%%%%%%%%%%%%%%%%%%%%%%%%%%%%%%%%%%%%%%

Throughout the paper, we fix a prime number $p$ and an algebraically closed 
field $\R$ of characteristic different from $p$. 

All representations are supposed to be smooth representations on $\R$-vector 
spaces. 
If $\G$ is a topo\-logical group, we write $\Rr(\G)$ for the abelian category of all 
representations of $\G$ and $\Irr(\G)$ for 
the set of all isomorphism classes of irreducible representations of $\G$. 
A \textit{character} of $\G$ is a
homomorphism from $\G$ to $\R^{\times}$ with open kernel.

For~$\G$ the group of points of a connected reductive group over either 
a finite field of characteristic~$p$ or a nonarchimedean locally compact 
field of residual characteristic~$p$, and~$\P=\M\N$ a parabolic subgroup 
of~$\G$ together with a Levi decomposition, we will write~$\ip^\G_\P$ for the 
\emph{normalized} parabolic induction functor from~$\Rr(\M)$ to~$\Rr(\G)$, 
and~$\ia^\G_\P$ for the \emph{unnormalized} parabolic induction functor 
from~$\Rr(\M)$ to~$\Rr(\G)$; these coincide in the finite field case.

%%%%%%%%%%%%%%%%%%%%%%%%%%%%%%%%%%%%%%%%%%%%%%%%%%%%%%%%%%%%%%%%%%%%%%%%
\section{Extensions and blocks}
\label{S0}
%%%%%%%%%%%%%%%%%%%%%%%%%%%%%%%%%%%%%%%%%%%%%%%%%%%%%%%%%%%%%%%%%%%%%%%%

We begin with some general results which apply to connected reductive groups 
over both finite and nonarchimedean locally compact fields. In the finite 
case, we give some further results towards a block decomposition, in 
particular for the group~$\GL_n$; these will be needed in the nonarchimedean 
case later. 

\subsection{}
Let~$\G$ be the group of points of a connected reductive group over either 
a finite field of characteristic~$p$ or a nonarchimedean locally compact 
field of residual characteristic~$p$.  

\begin{defi}
\label{DefSCP}
An irreducible representation~$\pi$ of~$\G$ is~\emph{supercuspidal} if it does 
not appear as a subquotient of any representation of the form 
$\ip^\G_\P(\tau)$, where $\P$ is a proper parabolic subgroup of $\G$ with Levi 
component $\M$ and $\tau$ is an \textit{irreducible} representation of $\M$. 

A \textit{supercuspidal pair} of $\G$ is a pair $(\M,\vr)$ made of a Levi 
subgroup $\M\subseteq\G$ and an irreducible supercuspidal representation 
$\vr$ of $\M$.

For~$\pi$ an irreducible representation of~$\G$, the \emph{supercuspidal 
support} of~$\pi$ is the set:
\begin{equation*}
\scusp(\pi)
\end{equation*} 
of supercuspidal pairs~$(\M,\vr)$ 
of~$\G$ such that~$\pi$ occurs as a subquotient of~$\ip^\G_\P(\vr)$, for some 
parabolic subgroup~$\P$ with Levi component~$\M$.
\end{defi}

\begin{rema}
In the finite field case, the word \emph{irreducible} may be omitted from the 
definition of supercuspidal (see Proposition~\ref{DDSCfini}); we will see 
that, for~$\G$ an inner form of~$\GL_n$ over a nonarchimedean locally compact 
field, the same is true (see Proposition~\ref{DDSC}).
\end{rema}

Similarly, an irreducible representation~$\pi$ of~$\G$ is~\emph{cuspidal} 
if it does not appear as a quotient of any representation of the 
form~$\ip^\G_\P(\tau)$, and we have the notion of \emph{cuspidal pair} 
and \emph{cuspidal support}~$\cusp(\pi)$.  
It is known that the cuspidal 
support~$\cusp(\pi)$ consists of a single~$\G$-conjugacy class of 
cuspidal pairs (\cite[Théorème 2.1]{MS2})
but there is no such general result for supercuspidal support; 
indeed, it is not even known that the possible supercuspidal supports form a 
partition of the set of supercuspidal pairs.

In this section, we make the following hypotheses:
\begin{itemize}
\item[(H1)] 
for~$\pi,\pi'$ irreducible representations of~$\G$,
  if~$\scusp(\pi)\cap\scusp(\pi')\ne\emptyset$ 
  then~$\scusp(\pi)=\scusp(\pi')$. 
%\item[(H1)] for any irreducible representation~$\pi$ of~$\G$, the 
%set~$\scusp(\pi)$ consists of a single~$\G$-conjugacy class of 
%supercuspidal pairs;
\item[(H2)] for supercuspidal pairs~$(\M,\vr)$, $(\M,\vr')$ of~$\G$, if the 
~space $\Ext^{i}_\M(\vr',\vr)$ is nonzero for some~$i\>0$, 
then~$\vr'\simeq\vr$; 
\end{itemize}

\begin{prop}\label{Extii}
Assume hypotheses {\rm(H1)} and {\rm (H2)} are satisfied. 
Let $\pi$ and $\pi'$ be irreducible representations of $\G$ 
with unequal supercuspidal supports.  
Then $\Ext^i_\G(\pi',\pi)=0$ for all $i\>0$.
\end{prop}

The idea of computing all the $\Ext^i$
rather than $\Ext^1$ only (which allows us to reduce to the case where 
$\pi,\pi'$ are supercuspidal) comes from 
Emerton--Helm~\cite[Theorem 3.2.13]{EH}.

\begin{proof}
Let $\pi$ and $\pi'$ be irre\-ducible representations of $\G$ 
with unequal supercuspidal supports. 

\begin{lemm}
\label{cnc}
Assume that $\pi'$ is cuspidal and $\pi$ is not. 
Then we have $\Ext^i_\G(\pi',\pi)=0$ for all $i\>0$.
\end{lemm}

\begin{proof}
The proof is by induction on $i$,
the case where $i=0$ being immediate.
Let us embed $\pi$ in $\ip^\G_\P(\tau)$ 
with $\tau$ an irreducible cuspidal representation of a proper Levi 
subgroup $\M$ and $\P$ a para\-bo\-lic subgroup of Levi component 
$\M$ and unipotent radical $\N$.
We have an exact sequence
\begin{equation*}
\Ext^{i-1}_\G(\pi',\xi)\to\Ext^{i}_\G(\pi',\pi)\to\Ext^{i}_\G(\pi',\ip^\G_\P(\tau)),
\end{equation*}
where $\xi$ is the quotient of $\ip^\G_\P(\tau)$ by $\pi$. 
Since $\pi,\pi'$ have unequal supercuspidal supports, 
we have, by the inductive hypothesis, 
$\Ext^{i-1}_\G(\pi',\l)=0$ for all the irreducible subquotients $\l$ of $\xi$, 
thus we have $\Ext^{i-1}_\G(\pi',\xi)=0$.
By~\cite[I.A.2]{Vlivre}, we have an isomorphism:
\begin{equation*}
\Ext^{i}_\G(\pi',\ip^\G_\P(\tau))\simeq\Ext^{i}_\M(\pi'_\N,\tau)=0
\end{equation*}
(where $\pi'_\N$ is the Jacquet module of $\pi'$ with respect to $\P=\M\N$). 
This gives us $\Ext^{i}_\G(\pi',\pi)=0$ as expected. 
\end{proof}

In the case where $\pi$ is cuspidal and $\pi'$ is not, we reduce to 
Lemma~\ref{cnc} by taking contragredients.
Indeed, we have:
\begin{equation*}
\Ext^{i}_\G(\pi',\pi)\simeq\Ext^{i}_\G(\pi^\vee,\pi'^\vee)
\end{equation*}
and this is $0$ by the previous case.
We now treat the case where $\pi$ and $\pi'$ are both cuspidal. 

\begin{lemm}
\label{cnsc}
Assume that $\pi$ is not supercuspidal. 
Then $\Ext^i_\G(\pi',\pi)=0$ for all $i\>0$.
\end{lemm}

\begin{proof}
The proof is by induction on $i$, the case where $i=0$ being immediate.
By assumption, $\pi$ occurs as a subquotient of $\ip^\G_\P(\tau)$, 
with $\tau$ an irreducible supercuspidal representation of a proper Levi 
subgroup $\M$ and $\P$ a para\-bo\-lic subgroup of Levi component 
$\M$ and unipotent radical $\N$.

Let~$\V$ be the minimal subrepresentations of~$\X=\ip^\G_\P(\tau)$ such 
that ~$\pi$ is a (sub)quotient of~$\V$, and let 
$\W\subseteq\V$ be a subrepresentation such that 
$\V/\W\simeq\pi$; thus~$\pi$ is not a subquotient of~$\W$. Denote 
by~$k=k(\pi)$ the number of irreducible cuspidal subquotients of~$\W$. Now we
proceed by induction on~$k$, noting that any irreducible subquotient~$\pi''$ of~$\W$ must have~$k(\pi'')\le k(\pi)-1$.

We have an exact sequence:
\begin{equation*}
\Ext^{i-1}_\G(\pi',\X/\V)\to\Ext^i_\G(\pi',\V)
\to\Ext^{i}_\G(\pi',\ip^\G_\P(\tau))=0.
\end{equation*}
We claim that $\Ext^{i-1}_\G(\pi',\l)=0$ for all the irreducible subquotients 
$\l$ of $\X$.  
Indeed, this follows from Lemma~\ref{cnc} if $\l$ is not cuspidal and from the 
inductive hypothesis (on~$i$) if $\l$ is a cuspidal irreducible 
subquotient of $\X$.
This gives us $\Ext^{i-1}_\G(\pi',\X/\V)=0$, and it 
follows from the above exact sequence that $\Ext^i_\G(\pi',\V)=0$.
Now we have an exact sequence:
\begin{equation*}
0=\Ext^{i}_\G(\pi',\V)\to\Ext^i_\G(\pi',\pi)\to\Ext^{i+1}_\G(\pi',\W).
\end{equation*}
If~$k=0$ then all the irreducible subquotients of $\W$ are non-cuspidal 
and Lemma~\ref{cnc} implies that we have $\Ext^{i+1}_\G(\pi',\W)=0$; thus $\Ext^i_\G(\pi',\pi)=0$, which completes the base step of the induction on~$k$. For the general case, since every irreducible subquotient~$\pi''$ of~$\W$ is either non-cuspidal or has~$k(\pi'')<k$, we again have~$\Ext^{i+1}_\G(\pi',\W)=0$, by Lemma~\ref{cnc} and the inductive hypothesis (on~$k$).
\end{proof}

We have the same result when $\pi'$ is not supercuspidal,
by taking contragredients as above.

\begin{coro}
Suppose that~$\pi$,~$\pi'$ are cuspidal. Then~$\Ext_{\G}^i(\pi',\pi)=0$ for all~$i\ge 0$.
\end{coro}

\begin{proof}
If either~$\pi$ or~$\pi'$ is not supercuspidal then the result follows from Lemma~\ref{cnsc}. If both are supercuspidal then this is the hypothesis (H2).
\end{proof}

We now treat the general case. 
The proof is by induction on $i$,
the case $i=0$ being trivial.
We have an exact sequence:
\begin{equation*}
0=\Ext^{i-1}_\G(\pi',\ip^\G_\P(\tau))\to\Ext^{i}_\G(\pi',\pi)\to\Ext^{i}_\G(\pi',\ip^\G_\P(\tau))
\simeq\Ext^{1}_\M(\pi'_\N,\tau)
\end{equation*}
where $\pi$ embeds in $\ip^\G_\P(\tau)$ with $\tau$ an irreducible cuspidal 
representation of $\M$.
From the cuspidal case, we have $\Ext^{i}_\M(\s,\tau)=0$ for all irreducible 
representations $\s$ of $\M$ that are nonisomorphic to $\tau$.
If we prove that $\tau$ does not appear as a subquotient of $\pi'_\N$, then we 
will get $\Ext^{i}_\M(\pi'_\N,\tau)=0$ and the result will follow. 

Assume that $\tau$ appears as a subquotient of $\pi'_\N$.
Let $\l'$ be an irreducible supercuspidal
repre\-sentation of a Levi subgroup~$\M'$ such that $\pi'$ occurs as a subquotient 
of $\ip^\G_{\P'}(\l')$, for some parabolic subgroup~$\P'$ with Levi component~$\M'$.
By the Geometric Lemma (see for example~\cite[(1.3)]{MS2}), there is a 
permutation matrix $w$ such that $\tau$ occurs in: 
\begin{equation*}
\ip^\M_{\M\cap\P'^{w}}(\l'^{w}).
\end{equation*}
By replacing $\l'$ by $\l'^{w}$, we may assume that $w=1$, so that $\tau$ 
occurs in $\ip^\M_{\M\cap\P'}(\l')$.
By applying $\ip^\G_\P$, we deduce that $\pi$ occurs in $\ip^\G_{\P'}(\l')$.
This contradicts the fact that~$\pi,\pi'$ have unequal supercuspidal 
supports. 
\end{proof}

\begin{prop}
\label{DecBySuperCusp}
Assume hypotheses {\rm (H1)} and {\rm (H2)} are satisfied.
Let $\V$ be a representation of $\G$ of finite length. 
There is a decomposition:
\begin{equation*}
\V=\V_1\oplus\dots\oplus\V_r
\end{equation*}
of $\V$ as a direct sum of subrepresentations where, for each 
$i\in\{1,\dots,r\}$, all irreducible subquotients of $\V_i$ have 
the same supercuspidal support.
\end{prop}

\begin{proof}
Write $n$ for the length of $\V$ and $r$ for the number of distinct 
sets~$\scusp(\pi)$, for~$\pi$ an irreducible subquotient of $\V$.
We may and will assume that $r>1$.
The proof is by induction on $n$.

Since $r\<n$, the minimal case with $r>1$ is $r=n=2$.
Assume we are in this case. 
Then the result follows from Proposition~\ref{Extii} with~$i=1$. 

Assume now that $n>2$.
Let $\om_0$ be the supercuspidal support of an irreducible sub\-re\-presen\-tation of 
$\V$ and $\V_0$ be the maximal subrepresentation of $\V$ all of whose 
irreducible sub\-quo\-tients have supercuspidal support~$\om_0$.
By the inductive hypothesis, $\V/\V_0$ decomposes as a direct sum:
\begin{equation*}
\W_1\oplus\dots\oplus\W_s
\end{equation*}
of nonzero subrepresentations, 
with $s\<r$ and where, for each $i\in\{1,\dots,s\}$, 
there is a supercuspidal support $\om_i$ such that all 
irreducible subquotients of $\W_i$ have supercuspidal support $\om_i$. 
If there is $i\>1$ such that $\om_i=\om_0$, 
then the preimage of $\W_i$ in $\V$ would 
contradict the maximality of $\V_0$.
Thus we have $\om_0\notin\{\om_1,\dots,\om_s\}$ and it follows that 
$r=s+1$.

\begin{lemm}
For each $i\in\{1,\dots,s\}$, 
there is an injective homomorphism $f_i:\W_i\to\V$.
\end{lemm}

\begin{proof}
For $i\in\{1,\dots,s\}$, write $\Y_i$ for the preimage of $\W_i$ in $\V$.  
If $\Y_i\neq\V$, then it follows from the inductive hypothesis that $\Y_i$ 
decomposes into the direct sum of $\V_0$ and a subrepresentation isomorphic to 
$\W_i$. 

Assume now that $\Y_i=\V$, thus $r=2$ and $i=1$.
By passing to the contragredient if necessary (and thus exchang\-ing the
roles of $\V_0$ and $\V_1$) we may assume that $\V_0$ is reducible.
Let $\pi$ denote an irreducible subrepresentation of $\V_0$.
By the inductive hypothesis, $\V/\pi$ has a direct summand isomorphic 
to $\W_1$.
Its preimage in $\V$ is denoted $\X_1$ and we can apply the inductive 
hypothesis to it. 
Thus $\W_1$ embeds in $\V$.
\end{proof}

We thus have injective homomorphisms $f_1,\dots,f_s$, and write $f_0$ for 
the canonical inclusion of $\V_0$ in $\V$.
We write $\V_i=f_i(\W_i)$ for all $i\in\{0,\dots,s\}$ and claim that we have:
\begin{equation*}
\V=\V_0\oplus\dots\oplus\V_s.
\end{equation*}
Indeed, we have a homomorphism:
\begin{equation*}
f:\V_0\oplus\big(\bigoplus\limits_{i=1}^{s}\W_i\big)=\X\to\V.
\end{equation*}
Since $\X$ and $\V$ have the same length, it is enough to prove that $f$ is 
injective. 

\begin{lemm}
We have:
\begin{equation*}
%\label{EXC}
\Ker(f)=(\Ker(f)\cap\V_0)\oplus\(\bigoplus\limits_{i=1}^{s}(\Ker(f)\cap\W_i)\).
\end{equation*}
\end{lemm}

\begin{proof}
Since $f$ is nonzero, we have $\Ker(f)\subsetneq\V$, thus we can apply the 
inductive hypothesis to $\Ker(f)$.
The decomposition that we obtain is the right hand side of the expected 
equality. 
\end{proof}

Since $f_1,\dots,f_s$ are injective, we get $\Ker(f)\cap\W_i=0$ for all 
$i\in\{1,\dots,s\}$.
Thus $f$ is injective and the result is proved. 
\end{proof}

%%%%%%%%%%%%%%%%%%%%%%%%%%%%%%%%%%%%%%%%%%%%%%%%%%%%%%%%%%%%%%%%%%%%%%%%
\subsection{}
Now we specialize to the case that~$\G$ is a connected reductive group 
over a finite field. We begin with a general result which is independent 
of the hypotheses (H1) and (H2).

\begin{prop}
\label{DDSCfini}
Let $\P$ be a proper parabolic subgroup of $\G$ 
and $\s$ be a representation of a Levi component $\M$ of $\P$.
Then $\ip^\G_\P(\s)$ has no supercus\-pi\-dal irreducible 
subquotient.
\end{prop}

\begin{proof}
When $\s$ is irreducible, the result follows from the definition of a 
supercus\-pi\-dal representation.
Assume $\E=\ip^\G_\P(\s)$ contains a supercus\-pi\-dal irreducible 
subquotient $\pi$, and let us fix a projective envelope $\Pi$ of $\pi$ 
in $\Rr(\G)$. 
By~\cite[Proposition~2.3]{Hiss}, all its irreducible subquotients 
are cuspidal (indeed, this is a characterization of supercuspidal 
representations).
Let $\V$ be a sub\-representation of $\E$ having a quotient isomorphic 
to $\pi$.
As $\Pi$ is projective, we get a nonzero homo\-mor\-phism from $\Pi$ to $\V$,
whence it follows that some irreducible subquotient $\pi'$ of~$\Pi$ occurs as 
a subrepresentation of $\V$, thus of $\E$.
By Frobenius reciprocity, we get that the space $\pi'_\N$ 
of $\N$-coinvariants, where $\N$ is the unipotent radical of $\P$, 
is nonzero, contradicting the cuspidality of~$\pi'$.
\end{proof}

Let $\R[\G]$ be the group algebra of $\G$ over $\R$.
It decomposes as a direct sum:
\begin{equation*}
\R[\G]=\B_1\oplus\dots\oplus\B_t
\end{equation*}
of indecomposable two-sided ideals, called blocks of $\R[\G]$.
This corresponds to a decomposition:
\begin{equation*}
1=e_1+\dots+e_t
\end{equation*}
of the identity of $\R[\G]$ as a sum of indecomposable central idempotents. 
This implies a decomposition:
\begin{equation*}
\Rr(\G)=\Rr(\B_1)\oplus\dots\oplus\Rr(\B_t)
\end{equation*}
of the category $\Rr(\G)$ of $\R$-representations of $\G$ 
(that is, of left $\R[\G]$-modules)
into the direct 
sum of the subcategories $\Rr(\B_i)$, $i\in\{1,\dots,t\}$, where $\Rr(\B_i)$ 
is made of all representations $\V$ of $\G$ such that $e_i\V=\V$.

\begin{lemm}
Assume that hypotheses {\rm (H1)} and {\rm (H2)} are satisfied.
Let $\V\in\Rr(\B_i)$ for some $i\in\{1,\dots,t\}$.
Then all the irreducible subquotients of $\V$ have the same supercuspidal 
support.
\end{lemm}

\begin{proof}
If we apply Proposition~\ref{DecBySuperCusp} to the regular representation 
$\R[\G]$, which has finite 
length, we get that all the irreducible subquotients of $\B_i$ have the same 
supercuspidal support.  
Since all the irreducible subquotients of $\V$ are isomorphic to subquotients 
of $\B_i$, we get the result. 
\end{proof}

We deduce the following decomposition theorem. 

\begin{theo}
\label{DecFiniteGp}
Assume hypotheses {\rm (H1)} and {\rm (H2)} are satisfied.
Let $\V$ be a representation of $\G$.
For any supercuspidal support $\om$ of $\G$, let $\V(\om)$ 
denote the maximal subrepresentation of $\V$ all of 
whose irreducible subquotients have supercuspidal 
support $\om$. 
Then we have:
\begin{equation*}
\V=\bigoplus\limits_{\om}\V(\om).
\end{equation*}
\end{theo}

%%%%%%%%%%%%%%%%%%%%%%%%%%%%%%%%%%%%%%%%%%%%%%%%%%%%%%%%%%%%%%%%%%%%%%%%
\subsection{}

Finally, we specialize to the case where~$\G$ is the finite 
group~$\GL_n(q)$, with~$n\>1$ an integer and~$q$ a power of~$p$. 
In this case, it is known (\cite{James}) that
the supercuspidal support consists of a single~$\G$-conjugacy class of 
supercuspidal pairs, so~(H1) is satisfied. We prove that~(H2) is also 
satisfied.

\begin{lemm}
Let $\pi,\pi'$ be irreducible supercuspidal representations of $\G$ such that 
the space $\Ext^{i}_\G(\pi',\pi)$ is nonzero for some $i\>0$.
Then $\pi'\simeq\pi$.
\end{lemm}

\begin{proof}
The proof is by induction on $i$,
the case $i=0$ being trivial.
Let us fix a projective envelope $\Pi$ of $\pi$ in $\Rr(\G)$. 
By \cite[III.2.9]{Vlivre}, 
it has finite length, and all its irreducible subquotients are 
isomorphic to $\pi$. 
Consider the exact sequence:
\begin{equation*}
0\to\Pi_1\to\Pi\to\pi\to0
\end{equation*}
where $\Pi_1$ is the kernel of $\Pi\to\pi$. 
Then we have an exact sequence:
\begin{equation*}
\Ext^{i-1}_\G(\pi',\pi)\to\Ext^{i}_\G(\pi',\Pi_1)\to\Ext^{i}_\G(\pi',\Pi).
\end{equation*}
By the inductive hypothesis, we have $\Ext^{i-1}_\G(\pi',\pi)=0$.
Since $\Pi$ is projective in $\Rr(\G)$, we have $\Ext^{i}_\G(\pi',\Pi)=0$.
It follows that we have $\Ext^{i}_\G(\pi',\Pi_1)=0$.
Since all irreducible subquotients of $\Pi_1$ are 
isomorphic to $\pi$, we can consider an exact sequence:
\begin{equation*}
0\to\Pi_2\to\Pi_1\to\pi\to0
\end{equation*}
where $\Pi_2$ is the kernel of $\Pi_1\to\pi$. 
By induction, we define a finite decreasing sequence:
\begin{equation*}
\Pi=\Pi_0\supsetneq\Pi_1\supsetneq\Pi_2\supsetneq\dots
\supsetneq\Pi_r\supsetneq\Pi_{r+1}=0
\end{equation*}
of subrepresentations of $\Pi$ such that $\Pi_j/\Pi_{j+1}\simeq\pi$ 
and $\Ext^{i}_\G(\pi',\Pi_j)=0$ for all $j\>0$.
For $j=r$, we get the expected result. 
\end{proof}

In particular, since every Levi subgroup of~$\G$ is isomorphic to a product of smaller general linear groups, the hypothesis (H2) is satisfied and the conclusion of Theorem~\ref{DecFiniteGp} holds for~$\G$.

As a corollary, we will need a weaker result in Section~\ref{S7}, in which we allow for the action of a Galois group. Fix~$\Ga$ be a group of automorphisms of the finite field~$\FF_q$.

\begin{defi}\label{DefELS}
Let~$(\M,\vr)$ be a supercuspidal pair of~$\G$, with
\[
\M\simeq\GL_{n_1}(q)\tdt\GL_{n_r}(q),\qquad 
\vr\simeq\vr_1\odo\vr_r.
\]
The \textit{equivalence class} of~$(\M,\vr)$ is the set, 
denoted~$[\M,\vr]$, of all supercuspidal pairs $(\M',\vr')$ of $\G$ 
for which there are elements~$\g_{i}\in\Ga$,
for each~$i=1,\ldots,r$, such that $(\M',\vr')$ is $\G$-conjugate 
to~$\(\M,\bigotimes_{i=1}^r\vr_{i}^{\g_{i}}\)$.
\end{defi}

\begin{coro}\label{cor:decomp}
Let $\V$ be a representation of $\G$ and, for any equivalence class 
of supercuspidal pairs $[\om]$, 
write $\V[\om]$ for the maximal subrepresentation of $\V$ all of 
whose irreducible sub\-quotients have supercuspidal 
support contained in $[\om]$. 
Then $\V$ decomposes into the direct sum of the $\V[\om]$, 
where $[\om]$ ranges over the set of equivalence classes 
of supercuspidal pairs of~$\G$.
\end{coro}

%%%%%%%%%%%%%%%%%%%%%%%%%%%%%%%%%%%%%%%%%%%%%%%%%%%%%%%%%%%%%%%%%%%%%%%%
\sectionNN*{Further notation}
%%%%%%%%%%%%%%%%%%%%%%%%%%%%%%%%%%%%%%%%%%%%%%%%%%%%%%%%%%%%%%%%%%%%%%%%

Throughout the rest of the paper, we fix  
a nonarchimedean locally compact field $\F$ of residue characteristic $p$. 
All $\F$-algebras are supposed to be fi\-ni\-te-dimensional with a unit. 
By an \textit{$\F$-division algebra} we mean a central 
$\F$-al\-gebra which is a division algebra.

For $\K$ a finite extension of $\F$, or more generally a division algebra over
a finite extension of $\F$, we denote by $\Oo_{\K}$ its ring of integers, 
by $\p_{\K}$ the maximal ideal of $\Oo_{\K}$ and by $\mathfrak{k}_\K$
its residue field.

For $\A$ a simple central algebra over a finite extension $\K$ of $\F$, we 
denote by $\N_{\A/\K}$ and $\tr_{\A/\K}$ respectively the reduced norm and 
trace of $\A$ over $\K$.

For $u$ a real number, we denote by 
$\lfloor{u}\rfloor$ the greatest integer which is smaller than or equal to 
$u$, that is its integer part.

A \textit{composition} of an integer $m\>1$ is a finite family 
of positive integers whose sum is $m$.

Given $\H$ a closed subgroup of a topo\-logical group $\G$ and $\s$ a 
representation 
of $\H$, write $\ind^\G_\H(\s)$ for the representation of $\G$ compactly 
induced from $\s$.
% and $\Hh(\G,\s)$ for its endomorphism algebra. 

We fix once and for all an additive character $\psi_\F:\F\to\mult\R$
that we assume to be trivial on $\p_\F$ but not on $\Oo_\F$. 

%%%%%%%%%%%%%%%%%%%%%%%%%%%%%%%%%%%%%%%%%%%%%%%%%%%%%%%%%%%%%%%%%%%%%%%%
\section{Preliminaries}
\label{S1}
%%%%%%%%%%%%%%%%%%%%%%%%%%%%%%%%%%%%%%%%%%%%%%%%%%%%%%%%%%%%%%%%%%%%%%%%

We fix an $\F$-division algebra $\D$, with reduced degree $d$.  
For all $m\>1$, we write $\A_{m}=\Mat_{m}(\D)$ and $\G_{m}=\GL_{m}(\D)$. 

Let $m\>1$ be a positive integer and write $\A=\A_{m}$ and $\G=\G_m$.
We will recall briefly the objects associated to the explicit construction 
of representations of $\G$; we refer to~\cite{VS1,VS2,VS3,SeSt1,SeSt2} for 
more details on the notions of simple stratum, character and type.

Recall that, for~$\P=\M\N$ a parabolic subgroup of $\G$ together with a Levi 
decomposition, we write $\ia^\G_\P$ for the unnormalized parabolic 
induction functor from $\Rr(\M)$ to $\Rr(\G)$.

\subsection{}

Recall (see \cite[Th\'eor\`eme 8.16]{MS2}) that, for~$\pi$ an irreducible 
representation of~$\G$, the supercuspidal support~$\scusp(\pi)$ consists of 
a single~$\G$-conjugacy class of supercuspidal pairs of~$\G$. 

\begin{defi}
The \textit{inertial class} of a supercuspidal pair $(\M,\vr)$ of $\G$ is the set, 
denoted by $[\M,\vr]_{\G}$, of all super\-cus\-pidal pairs $(\M',\vr')$ 
that are $\G$-conjugate to $(\M,\vr\chi)$ for some un\-ram\-ified 
character $\chi$ of $\M$. 
\end{defi}

\subsection{}

Let $\La$ be an $\Oo_\D$-lattice sequen\-ce of $\D^m$.
It defines an hereditary $\Oo_\F$-order $\AA(\La)$ of $\A$ and an 
$\Oo_\F$-lattice sequen\-ce:
\begin{equation*}
\aa_{k}(\La)=\{a\in\A\ |\ a\La(i)\subseteq\La({i+k}), \text{ for all } i\in\ZZ\}
\end{equation*}
of $\A$.
For $i\>1$, we write $\U_i(\La)=1+\aa_{i}(\La)$.
This defines a filtration $(\U_i(\La))_{i\>1}$ of the compact open subgroup 
$\U(\La)=\AA(\La)^\times$ of $\G$. 

Let $[\La,n,0,\b]$ be a simple stratum in $\A$ (see for instance \cite[\S1.6]{SeSt1}).
The element $\b\in\A$ gene\-rates a field extension $\F[\b]$ of $\F$, denoted 
$\E$, and we write $\B$ for its centralizer in $\A$. 
Attached to this stratum, there are two compact open subgroups:
\begin{equation*}
\J=\J(\b,\La),
\quad
\H=\H(\b,\La)
\end{equation*}
of $\G$. 
For all $i\>1$, we set:
\begin{equation*}
\J^i=\J^i(\b,\La)=\J\cap\U_i(\La),
\quad
\H^i=\H^i(\b,\La)=\H\cap\U_i(\La).
\end{equation*}
Together with the choice of $\psi_\F$, the simple stratum defines 
a finite set $\Cc(\La,0,\b)$ of characters of $\H^{1}$, 
called simple characters. 
We do not recall here the definition of these char\-acters, only the following
basic property. 
Write $\psi_\A=\psi_\F\circ\tr_{\A/\F}$ and, for $b\in\A$, set: 
\begin{equation*}
\psi_b:x\mapsto\psi_\A(b(x-1))
\end{equation*}
for all $x\in\A$. 
If $b\in\aa_{-k}(\La)$ for some $k\>1$, then $\psi_b$ defines a character on 
$\U_{\lfloor k/2\rfloor+1}(\La)$.  
Then any simple character~$\theta\in\Cc(\La,0,\b)$ 
satisfies~$\theta|_{\U_{\lfloor n/2\rfloor+1}(\La)}=\psi_\beta$. 

Given $\t$ a simple character attached to $[\La,n,0,\b]$, there is, up to 
isomorphism, a unique irre\-du\-cible representation $\n$ of 
$\J^{1}$ whose restriction to $\H^1$ contains $\t$. 
Moreover, the representation $\n$ extends to an irreducible representation of 
the group $\J$ that is intertwined by the whole of $\mult\B$. 
Such extensions of $\n$ to $\J$ are called $\b$-extensions. 

As $\B$ is a central simple $\E$-algebra, there are a positive integer $m'\>1$, 
an $\E$-division algebra $\D'$ and an isomorphism of $\E$-algebras $\Phi$ 
from $\B$ to $\Mat_{m'}(\D')$. 
Moreover, we can choose $\Phi$ so that $\Phi(\AA(\La)\cap\B)$ is a standard 
order, that is, it is contained in $\Mat_{m'}(\Oo_{\D'})$ and its reduction
mod $\p_{\D'}$ is upper block triangular.
Since $\J=(\U(\La)\cap\mult\B)\J^1$, we thus have group isomorphisms: 
\begin{equation*}
\J/\J^1\simeq
(\U(\La)\cap\mult\B) / (\U_1(\La)\cap\mult\B)
\simeq
\GL_{m'_1}(\kk_{\D'})\times\dots\times\GL_{m'_r}(\kk_{\D'})
\end{equation*}
for suitable positive integers $m'_1,\dots,m'_r$.
It allows us to identify these groups and 
we denote by $\Gg$ the latter group.

A \textit{simple type} attached to $[\La,n,0,\b]$ is an irreducible representation $\l$ of 
$\J$ of the form $\k\otimes\s$, where $\k$ is a $\b$-extension and $\s$ 
is an irreducible representation of $\J$ trivial on $\J^1$ which identifies 
with a cuspidal representation of $\Gg$ of the form 
$\tau\odo\tau$ where $\tau$ is a cuspidal representation of 
$\GL_{m'/r}(\kk_{\D'})$ (this implies $m'_1=\dots=m'_r=m'/r$).
When the representation $\tau$ is supercuspidal, $\l$ is called a \textit{simple supertype}. 

We introduce the following useful definition. 

\begin{defi}
A simple character (or a $\b$-extension, or a simple type) 
is said to be \textit{maximal} if $\U(\La)\cap\mult\B$ is a maximal compact 
open subgroup in $\mult\B$. 
\end{defi}

%%%%%%%%%%%%%%%%%%%%%%%%%%%%%%%%%%%%%%%%%%%%%%%%%%%%%%%%%%%%%%%%%%%%%%%%
\section{An abstract $\KM$-functor}
%%%%%%%%%%%%%%%%%%%%%%%%%%%%%%%%%%%%%%%%%%%%%%%%%%%%%%%%%%%%%%%%%%%%%%%%
\label{SST}

The functor~$\KM$ was first introduced in the split case for complex 
representations in~\cite{SZ}, where it was used just for simple 
types. In~\cite{MS1} this was generalized to apply to any~$\G$ in the modular 
case. %, and also to semisimple types 
It will be a main tool for us, but we will need several variants of it so it
is convenient to give a general setup which applies to all situations. 

Let $\P=\M\N$ be a parabolic subgroup of $\G$, to\-ge\-ther with a Levi 
decomposition. Given $g\in\G$, $\K$ a compact open subgroup of $\G$ and $\pi$ 
a representation of $\M$, write:
\begin{equation*}
\Ind_\P^{\P g\K}(\pi) = \{f\in \ia^\G_\P(\pi)\ |\ 
\text{$f$ is supported in $\P g\K$}\}.
\end{equation*}
% $\Ind_\P^{\P g\K}(\pi)$ 
% for the subspace of 
% $\ia^\G_\P(\pi)$ consisting of all functions supported in $\P g\K$.  
This defines a functor from $\Rr(\M)$ to $\Rr(\K)$ denoted $\Ind_\P^{\P g\K}$.  

We have the following easy but useful lemma.

\begin{lemm}
\label{L0}
Let $\K$ be a compact open subgroup of $\G$.
For all representation $\pi$ of $\M$ and all $g\in\G$, 
there is an isomorphism:
\begin{equation*}
\Ind_\P^{\P g\K}(\pi)\simeq\Ind_{\K\cap\P^g}^{\K}(\pi^g)
\end{equation*}
of representations of $\K$, where $\P^g,\pi^g$ 
denote the conjugates of $\P,\pi$ by $g$.
\end{lemm}

\begin{proof}
The isomorphism is given by $f\mapsto f_g$, 
where $f_g(k)=f(gk)$ for all $k\in\K$.
\end{proof}

Now we are given a compact open subgroup~$\textbf{\textsf{J}}$ of~$\G$, 
together with a normal pro-$p$ subgroup~$\textbf{\textsf{J}}^1$, and an irreducible 
representation~$\bk$ of~$\textbf{\textsf{J}}$.
We define a functor:
\begin{equation*}
%\label{defKM}
\KM_{\bk}:\pi\mapsto\Hom_{\textbf{\textsf{J}}^1}(\bk,\pi)
\end{equation*}
from $\Rr(\G)$ to $\Rr(\textbf{\textsf{J}}/\textbf{\textsf{J}}^1)$, by making 
$\textbf{\textsf{J}}$ act on $\KM_{\bk}(\pi)$ by the formula: 
\begin{equation*}
x\cdot f=\pi(x)\circ f\circ \bk(x)^{-1} 
\end{equation*}
for all $x\in\textbf{\textsf{J}}$ and $f\in\KM_{\bk}(\pi)$.  Note that 
$\textbf{\textsf{J}}^1$ acts trivially.  Since $\textbf{\textsf{J}}^1$ is a 
pro-$p$-group, this functor is exact, and it sends admissible representations 
of $\G$ to finite dimensional representations of 
$\textbf{\textsf{J}}/\textbf{\textsf{J}}^1$.  

\begin{prop}
\label{Annonce}
Let $g\in\G$. The following are equivalent:
\begin{enumerate}
\item the functor $\KM_{\bk}\circ\Ind_\P^{\P g\textbf{\textsf{J}}}$ is nonzero on $\Rr(\M)$;
\item the functor $\KM_{\bk}\circ\Ind_\P^{\P g\textbf{\textsf{J}}}$ is nonzero on $\Irr(\M)$;
\item $\Hom_{\textbf{\textsf{J}}^1\cap\N^g}(\bk,1)\neq0$ 
(or, equivalently,~$\k$ has a non-zero 
$\textbf{\textsf{J}}^1\cap\N^g$-fixed vector).
\end{enumerate}
\end{prop}

\begin{proof}
Given $\pi\in\Rr(\M)$, by Lemma~\ref{L0} we have an isomorphism: 
\begin{equation*}
\Ind_\P^{\P g\textbf{\textsf{J}}}(\pi)\simeq
\Ind_{\textbf{\textsf{J}}\cap\P^g}^{\textbf{\textsf{J}}}(\pi^g)
\end{equation*}
of representations of $\textbf{\textsf{J}}$. Applying Mackey's formula and Frobenius reciprocity, 
and writing~$\bn$ for the restriction of~$\bk$ to~$\textbf{\textsf{J}}^1$, we get: 
\begin{equation*}
\KM_{\bk}(\Ind_\P^{\P g\textbf{\textsf{J}}}(\pi))\simeq
\bigoplus\limits_{x\in(\textbf{\textsf{J}}^{}\cap\P^g)\backslash\textbf{\textsf{J}}^{}/\textbf{\textsf{J}}^1}
\Hom_{\textbf{\textsf{J}}^1\cap\P^{gx}}(\bn,\pi^{gx}).
\end{equation*}
As $\bn$ is normalized by $\textbf{\textsf{J}}$, this implies that:
\begin{equation*}
\KM_{\bk}(\Ind_\P^{\P g\textbf{\textsf{J}}}(\pi))\neq0
\quad
\Leftrightarrow
\quad
\Hom_{\textbf{\textsf{J}}^1\cap\P^{g}}(\bn,\pi^{g})\neq0.
\end{equation*}
As $\pi$ is trivial on $\N$, we have:
\begin{equation*}
\Hom_{\textbf{\textsf{J}}^1\cap\P^{g}}(\bn,\pi^{g})
\subseteq 
\Hom_{\textbf{\textsf{J}}^1\cap\N^{g}}(\bn,1)
\end{equation*}
Therefore, if~$\KM_{\bk}\circ\Ind_\P^{\P g\textbf{\textsf{J}}}$ is nonzero on $\Rr(\M)$, 
then $\Hom_{\textbf{\textsf{J}}^1\cap\N^g}(\bn,1)\neq0$. Thus (i) implies (iii), and it is clear that (ii) implies (i).

Now we assume that~$\Hom_{\textbf{\textsf{J}}^1\cap\N^g}(\bn,1)\neq0$ and write 
$\P'=\P^g$, $\N'=\N^g$, $\M'=\M^g$.
Define the compactly induced representation
\begin{equation*}
\V=\ind^{\P'}_{\textbf{\textsf{J}}^1\cap\P'}(\bn). 
\end{equation*}
For any $\pi\in\Rr(\M)$, as $\pi^g$ is trivial on $\N'$, we have
\begin{equation*}
\Hom_{\textbf{\textsf{J}}^1\cap\P'}(\bn,\pi^{g})
\simeq\Hom_{\P'}(\V,\pi^g)
\simeq\Hom_{\M'}(\V_{\N'},\pi^g),
\end{equation*}
where $\V_{\N'}$ denotes the space of $\N'$-coinvariants of $\V$. 
But
\begin{equation*}
\V_{\N'}
\simeq\bigoplus\limits_{l\in(\textbf{\textsf{J}}^1\cap\M')\backslash\M'}
\(\ind^{\N'}_{\N'\cap(\textbf{\textsf{J}}^1)^l}(\bn^l)\)_{\N'}
\simeq\bigoplus\limits_{l\in(\textbf{\textsf{J}}^1\cap\M')\backslash\M'}
(\bn^l)_{\N'\cap(\textbf{\textsf{J}}^1)^l},
\end{equation*}
by Shapiro's lemma, and the term corresponding to $l=1$ is nonzero.  
Thus $\V_{\N'}$ is nonzero and, moreover, it is of finite type since~$\V$ is 
of finite type and Jacquet functors preserve finite type.  Thus 
$(\V_{\N'})^{g^{-1}}$ has an irreducible quotient $\pi\in\Irr(\M)$ 
and~$\KM_{\bk}\circ\Ind_\P^{\P g\textbf{\textsf{J}}}(\pi)$ is nonzero. Hence 
(iii) implies (ii). 
\end{proof}

In some situations, we know more about the representation~$\bk$ and can 
conveniently rephrase the final condition of Proposition~\ref{Annonce}. 

\begin{coro}\label{cor:Annonce}
Write $\bn$ for the restriction of $\bk$ to $\BJ^1$, and 
suppose that we have a normal pro-$p$ subgroup~$\textbf{\textsf{H}}^1$ 
of~$\textbf{\textsf{J}}^1$ and a character~$\bt$ of~$\textbf{\textsf{H}}^1$ 
such that the restriction of~$\bn$ to~$\textbf{\textsf{H}}^1$ 
is~$\bt$-isotypic and that~$\bn$ is the unique irreducible 
repre\-sentation of~$\textbf{\textsf{J}}^1$ which contains~$\bt$. 
Then the conditions of Proposition~\ref{Annonce} are also equivalent to:
\begin{enumerate}
\setcounter{enumi}{3}
\item the character $\bt$ is trivial on $\textbf{\textsf{H}}^1\cap\N^{g}$.
\end{enumerate}
\end{coro}

\begin{proof} 
(iii) is equivalent to (iv) since~$\ind^{\textbf{\textsf{J}}^1}_{\textbf{\textsf{H}}^1}(\bt)$ is a 
finite sum of copies of~$\bn$ and the restriction of~$\bn$ to~$\textbf{\textsf{H}}^1$ 
is~$\bt$-isotypic.
\end{proof}

The usefulness of conditions (iii) and (iv) is that they do not depend on characteristic of the 
ground field~$\R$; that is, if~$\bk$ is a~$\overline\ZZ_\ell$-representation  
then $\Hom_{\textbf{\textsf{J}}^1\cap\N^g}(\bk,1)\neq0$ if and only if the same is true for the 
reduction modulo~$\ell$ of~$\bk$ (see~\cite[Lemme~5.7]{MS1}).

%%%%%%%%%%%%%%%%%%%%%%%%%%%%%%%%%%%%%%%%%%%%%%%%%%%%%%%%%%%%%%%%%%%%%%%%
\section{A lemma on simple characters}
%\label{S2}
%%%%%%%%%%%%%%%%%%%%%%%%%%%%%%%%%%%%%%%%%%%%%%%%%%%%%%%%%%%%%%%%%%%%%%%%

Let $\t$ be a simple character with respect to a simple stratum 
$[\La,n,0,\b]$ in $\A$.
Let $\P=\M\N$ be a parabolic subgroup of $\G$ together with a Levi 
decomposition. 
The purpose of this section is to show that, under certain conditions, the 
criterion of Corollary~\ref{cor:Annonce} is satisfied. 

Given a subset $\X$ of $\A$, write $\X^*$ for the set of $a\in\A$ such that 
$\psi_\A(ax)=1$ for all $x\in\X$. 

\begin{defi}% [\cite{SeSt1}, definition 5.1]
The pair $(\M,\P)$ is \textit{subordinate to} the simple stratum 
$[\La,n,0,\b]$ if the idempotents in $\A$ that correspond to $\M$ 
are in $\B$ and if there is an isomorphism 
$\Phi:\B\to\Mat_{m'}(\D')$ of $\E$-algebras 
such that $\Phi(\AA(\La)\cap\B)$ is a standard order 
and $\Phi(\P\cap\mult\B)$ is a standard parabolic subgroup 
corresponding to a composition of $m'$ finer than or equal to that of 
$\Phi(\AA(\La)\cap\B)$.
\end{defi}

Assume this is the case. 
For $k\>1$ and $i\in\ZZ$, write $\H^{k}=\H^{k}(\b,\La)$ and
$\aa_i=\aa_i(\La)$, and:
\begin{equation*}
% \mathfrak{n}_k=
\mathfrak{n}_k(\b,\La)=\{x\in\AA(\La)\ |\ \b x-x\b\in\aa_k\}.
\end{equation*}
Write $q$ for the greatest integer $i\<n$ such that 
$\mathfrak{n}_{1-i}(\b,\La)\subseteq\AA(\La)\cap\B+\aa_1$ and $s=\lfloor (q+1)/2\rfloor$.
For $k\>1$, set:
\begin{equation*}
\Om^{k}=\Om^k(\b,\La)=1+\aa_{k}\cap\mathfrak{n}_{k-q}(\b,\La)+\mathfrak{j}^{s}(\b,\La),
\end{equation*}
where $\mathfrak{j}^{s}=\mathfrak{j}^{s}(\b,\La)$ is defined by $\J^s=1+\mathfrak{j}^{s}(\b,\La)$. 
Write $\N^-$ for the unipotent radical opposite to $\N$ with respect to $\M$. 

\begin{lemm}
\label{charN1}
Let $g\in\U_1(\La)\cap\N^-$ and $0\<m<q$.
Assume that $\t$ is trivial on the inter\-sec\-tion 
$(\U_1(\La)\cap\mult\B)\H^{m+1}\cap\N^{g}$.  
Then $g\in(\U_1(\La)\cap\mult\B)\Om^{q-m}$. 
\end{lemm}

\begin{proof}
First note that it is enough to prove the result when $m\>\lfloor q/2\rfloor$.  
Indeed, if $m<\lfloor q/2\rfloor$, then 
the result for $\lfloor q/2\rfloor$ implies that:
\begin{equation*}
g\in(\U_1(\La)\cap\mult\B)\Om^{s}=\J^1(\b,\La)
=(\U_1(\La)\cap\mult\B)\Om^{q-m}. 
\end{equation*}
The proof is by in\-duc\-tion on both $q$ and $m$ with 
$\lfloor q/2\rfloor\<m<q$. 
Write $\nn$, $\p$ for the Lie algebras of $\N$, $\P$ in $\A$,
and also $\nn^-$ for that of $\N^-$.

Assume first that $q=n$. 
Then $g$ normalizes $\H^{m+1}=\U_{m+1}(\La)$. 
Since we have $\U_{m+1}(\La)\cap\N^{g}=(\U_{m+1}(\La)\cap\N)^{g}$, 
and since $\t$ is trivial on $\U_{m+1}(\La)\cap\N$, the condition on $\t$
implies that
\begin{equation*}
\t([g^{-1},1+y])=1,
\end{equation*}
for all $y\in\aa_{m+1}\cap\nn$. 
Recall that, for~$b,x\in\A$, we have~$\psi_b(x)=\psi_\A(b(x-1))$.

\begin{lemm}
%\label{LE1}
We have $\psi_{g\b g^{-1}-\b}(1+y)=1$ 
for all $y\in\aa_{m+1}\cap\nn$. 
\end{lemm}

\begin{proof}
Since $\lfloor q/2\rfloor\<m$, the restriction of $\t$ to $\H^{m+1}$ is given 
by $\psi_{\b}$. 
Now:
\begin{eqnarray*}
\psi_{\b}(g^{-1}(1+y)g)&=&\psi_{\A}(\b g^{-1}yg)\\
&=&\psi_{\A}(g\b g^{-1}y)\\
&=&\psi_{g\b g^{-1}}(1+y)
\end{eqnarray*}
for all $y\in\aa_{m+1}\cap\nn$, which gives us the desired result. 
\end{proof}

If we write $g=1+u$, with $u\in\aa_1\cap\nn^-$, this gives us:
\begin{equation*}
g\b g^{-1}-\b=-a_{\b}(u)g^{-1}
\in(\aa_{m+1}\cap\nn)^*=\aa_{-m}+\nn^*,
\end{equation*}
where $a_\b$ is the map $x\mapsto\b x-x\b$ from $\A$ to $\A$.
Note that, since $\nn$ is an $\F$-vector space, we have for all $a\in\A$: 
\begin{equation*}
\tr_{\A/\F}(a\nn)\subseteq\Ker(\psi_\F)
\quad\Leftrightarrow\quad
\tr_{\A/\F}(a\nn)=\{0\}.
\end{equation*}
It follows that $\nn^*=\p$. 
Together with the fact that $a_{\b}(u)g^{-1}\in\nn^-$ and $g\in\U_1(\La)$, 
we get: 
\begin{equation*}
a_{\b}(u)\in\aa_{-m}.
\end{equation*}
This gives us: 
\begin{equation*}
u\in\mathfrak{n}_{-m}(\b,\La)\cap\aa_1=(\AA(\La)\cap\B+\aa_{n-m})\cap\aa_1,
\end{equation*}
where the last equality follows from \cite[Proposition~2.29]{SeSt1}. 
But:
\begin{equation*}
\Om^{n-m}
=1+\aa_{n-m}\cap\mathfrak{n}_{-m}(\b,\La)+\mathfrak{j}^s(\b,\La)
=1+\aa_{n-m}+\aa_s=1+\aa_{n-m}.
\end{equation*}
We thus get the expected result. 

We now assume that $q<n$, and we fix a simple stratum 
$[\La,n,q,\g]$ that is equivalent to the pure stratum 
$[\La,n,q,\b]$. 
First assume that $m=q-1$ and write:
\begin{equation*}
\t|_{\H^q\cap\N^g}=\psi_{c}\t_{\g}=1,
\end{equation*}
where $c=\b-\g\in\aa_{-q}$ and $\t_\g\in\Cc(\La,q-1,\g)$.
Now write $g=1+u$. 

\begin{lemm}
The character $\psi_{c}$ is trivial on $\H^q\cap\N^g$. 
\end{lemm}

\begin{proof}
Let $x=g^{-1}yg\in\mathfrak{h}^q\cap\nn^g$,
where $\mathfrak{h}^{k}$ is defined for $k\>1$ 
by $\H^k=1+\mathfrak{h}^{k}$. 
Then:
\begin{eqnarray*}
\psi_{c}(1+x)
&=&\psi_\F(\tr_{\A/\F}(gcg^{-1}y))\\
&=&\psi_\F(\tr_{\A/\F}(cy))\psi_\F(\tr_{\A/\F}(-a_c(u)g^{-1}y))\\
&=&\psi_\F(\tr_{\A/\F}(-a_c(u)xg^{-1}))
\end{eqnarray*}
since $cy\in\nn$ has trace $0$. 
Now $c\in\aa_{-q}$ and $u\in\aa_1$ and $xg^{-1}\in\aa_q$.
Since $\psi_\F$ is trivial on $\p_\F$, we get the expected result. 
\end{proof}

Thus $\t_\g$ is trivial on $\H^q\cap\N^g$.
Note that $\H^q=\H^q(\g,\La)$. 
By the inductive hypothesis, we get:
\begin{equation*}
g\in(\U_1(\La)\cap\B_\g^\times)\Om^{q'-(q-1)}(\g,\La)
=(\U_1(\La)\cap\B_\g^\times)
\big(1+\aa_{q'-(q-1)}\cap\mathfrak{n}_{1-q}(\g,\La)+\mathfrak{j}^s(\g,\La)\big)
\end{equation*}
where $q'=-k_0(\g,\La)$ and $\B_\g$ is the centralizer of $\F[\g]$ in $\A$. 

The following lemma generalizes \cite[(8.1.12)]{BK}. 

\begin{lemm}
\label{LE2}
Let $[\La,n,m,\b]$ be a simple stratum in $\A$ and $\t\in\Cc(\La,m,\b)$
be a simple char\-acter.
Let $z\in\aa_{q-m}\cap\mathfrak{n}_{-m}(\b,\La)$ and $\vartheta$ be a character of 
$\H^{m}$ whose restriction to $\H^{m+1}$ is $\t$.  
Then $1+z$ normalizes $\H^{m}$ and 
$\vartheta^{1+z}=\vartheta\cdot\psi_{a_\b(z)}$. 
\end{lemm}

\begin{proof}
We follow the proof of~\cite[(8.1.12)]{BK}, replacing the results 
from~\cite{BK} used there by their analogues in~\cite{VS1,SeSt1}.  
% We replace \cite[(1.4.9)]{BK} by \cite[Proposition~2.29]{SeSt1}, 
% \cite[(3.3.1)]{BK} by \cite[Proposition~3.46]{VS1}, 
% \cite[(3.2.11)]{BK} by 
% by replacing the results of \cite{BK} by those of \cite{VS1,SeSt1}.
\end{proof}

If we apply Lemma \ref{LE2} to the stratum $[\La,n,q-1,\g]$, the simple
character $\t_\g$, the element $g^{-1}=1+u'$ and the character $\t$,
then $g$ normalizes $\H^{q-1}(\g,\La)=\H^{q-1}$ and $\H^q(\g,\La)=\H^q$, 
and we have: 
\begin{equation*}
\t^{1+u'}=\t\cdot\psi_{a_\g(u')}
\end{equation*}
on $\H^q$. 
Since $c\in\aa_{-q}$ and $u'\in\aa_1$, we have 
$\psi_{a_\g(u')}=\psi_{a_\b(u')}$ on $\H^q$.  
We thus get: 
\begin{equation*}
\t([g^{-1},1+y])=\psi_{a_\b(u')}(1+y)=\psi_\A(a_\b(u')y)=1
\end{equation*}
for all $y\in\mathfrak{h}^q\cap\nn$. 
We need the following lemma.

\begin{lemm}
We have $(\mathfrak{h}^q)^*=a_\b(\mathfrak{j}^s)+\aa_{1-q}$. 
\end{lemm}

\begin{proof}
It is straightforward to check that we have the containment~$\supseteq$, so 
suppose~$x\in(\mathfrak{h}^q)^*$.  We denote by~$\sr$ a tame corestriction on~$\A$ 
relative to~$\E/\F$ (see for 
example~\cite[D\'efinition~2.25]{SeSt1}). 
By~\cite[Proposition~2.27]{SeSt1},~$\sr(x)\in\aa_{1-q}\cap\B$ 
so, by~\cite[Proposition~2.29]{SeSt1}, there exists~$y\in\aa_{1-q}$ such 
that~$\sr(x)=\sr(y)$.  
Thus~$x-y\in(\mathfrak{h}^q)^*\cap\ker(\sr)$ and, again 
by~\cite[Proposition~2.27]{SeSt1}, there 
is~$z\in\aa_1\cap\mathfrak{n}_{1-q}(\b,\La)+\mathfrak{j}^s$ such 
that~$x-y=a_\b(z)$.  
Since~$a_\b(\aa_1\cap\mathfrak{n}_{1-q}(\b,\La))\subseteq\aa_{1-q}$,
the result follows. 
\end{proof}

Therefore we have:
\begin{equation*}
a_\b(u')\in(\mathfrak{h}^q)^*+\p
=a_\b(\mathfrak{j}^s)+\aa_{1-q}+\p.
\end{equation*}
As it is also in $\nn^-$, we get:
\begin{equation*}
a_\b(u')\in a_\b(\mathfrak{j}^s)+\aa_{1-q}.
\end{equation*}
This implies $u'\in\aa_1\cap\mathfrak{n}_{1-q}(\b,\La)+\mathfrak{j}^s$, thus $g\in\Om^1$.

Assume now that the result is true for some $m\<q-1$, 
and that $\t$ is trivial on $\H^{m}\cap\N^g$.
Then it is trivial on $\H^{m+1}\cap\N^g$.
From the inductive hypothesis, we thus get 
$g\in(\U_1(\La)\cap\mult\B)\Om^{q-m}$.
By Lemma \ref{LE2}, this implies that $g$ normalizes $\H^m$ and that:
\begin{equation*}
\t^{1+u'}=\t\cdot\psi_{a_\b(u')}
\end{equation*}
on $\H^m$, with $g^{-1}=1+u'$.
This implies:
\begin{equation*}
\t([g^{-1},1+y])=1
\end{equation*}
for all $y\in\mathfrak{h}^m\cap\nn$. 
Therefore:
\begin{equation*}
a_\b(u')\in 
((\mathfrak{h}^m)^*+\p)\cap\nn^-
=(a_\b(\mathfrak{j}^s)+\aa_{1-m}+\p)\cap\nn^-
\subseteq a_\b(\mathfrak{j}^s)+\aa_{1-m}. 
\end{equation*}
Thus there is $j\in\mathfrak{j}^s$ such that:
\begin{equation*}
u'+j\in\mathfrak{n}_{1-m}(\b,\La)\cap\aa_1.
\end{equation*}
From \cite[Proposition~2.29]{SeSt1} we have:
\begin{equation*}
\mathfrak{n}_{1-m}(\b,\La)=\AA(\La)\cap\B+\aa_{q-m+1}\cap\mathfrak{n}_{1-m}(\b,\La).
\end{equation*}
This implies the expected result, that is $g\in(\U_1(\La)\cap\mult\B)\Om^{q-m+1}$. 
\end{proof}

Continuing with the same notation, we will also need the following variant of 
Lemma~\ref{charN1}.  We put~$\H_\P^1=\H^1(\J^1\cap\N)$, which is a normal 
subgroup of~$\J^1$, and define the character~$\t_\P$ of~$\H_\P^1$ by 
\[
\t_\P(hj) = \t(h),
\]
for~$h\in\H^1$ and $j\in\J^1\cap\N$.  
By~\cite[Proposition~5.4]{SeSt1}, 
if we write $\J_\P^1=\H^1(\J^1\cap\P)$, 
the inter\-twin\-ing of the character~$\t_\P$ is~$\J_\P^1\B^\times\J_\P^1$.

\begin{coro}\label{coro:tP}
Let $g\in\U_1(\La)\cap\N^-$ and assume that $\t_\P$ is trivial on the 
inter\-sec\-tion $\H_\P^1\cap\N^{g}$. Then $g\in\J_\P^1$.  
\end{coro}

\begin{proof} 
Suppose that~$g\in\U_1(\La)\cap\N^-$ and~$\t_\P$ is trivial 
on~$\H_\P^1\cap\N^{g}$.  In particular, intersecting with~$\H^1$, we see 
that~$\t$ is trivial on~$\H^1\cap\N^{g}$ so, by Lemma~\ref{charN1}, we 
find~$g\in\J^1\cap\N^-$. Since~$g$ then normalizes~$\t$, we see that it also 
normalizes~$\t_\P$, so lies in~$\J_\P^1\B^\times\J_\P^1\cap\J^1=\J_\P^1$. 
\end{proof}

%%%%%%%%%%%%%%%%%%%%%%%%%%%%%%%%%%%%%%%%%%%%%%%%%%%%%%%%%%%%%%%%%%%%%%%%
\section{Parabolic induction and the functor $\KM$ in the simple case}
%%%%%%%%%%%%%%%%%%%%%%%%%%%%%%%%%%%%%%%%%%%%%%%%%%%%%%%%%%%%%%%%%%%%%%%%

The main result of this section is Theorem~\ref{MainTheoK}, 
which says that, in the simple case, the func\-tor~$\KM$ commutes with
parabolic induction; in the next section we will extend this result to the semisimple case.
This fact has been claimed in \cite{MS1} 
for rep\-res\-entations of fi\-ni\-te length (see \cite{MS1}, Pro\-po\-sition 5.9)
but it appears that the proof of \textit{ibid.}, Lemme 5.10 requires more 
details. 

We give a different proof here, based on our 
Lemma~\ref{charN1}, which works for all smooth representations and 
not only for rep\-res\-entations of finite length.

\subsection{}
\label{badm}

Let $[\Lamax,n,0,\b]$ be a simple stratum in $\Mat_{m}(\D)$
and assume that $\U(\Lamax)\cap\B^\times$ is a max\-imal compact open subgroup in $\B^\times$.
Let $\tmax$ be a simple character in $\Cc(\Lamax,0,\b)$ and 
$\kmax$ be a $\b$-extension of $\tmax$. 
We write $\jmax^{}=\J(\b,\Lamax)$ and $\jmax^1=\J^1(\b,\Lamax)$.
Let $\KM$ be the functor:
\begin{equation*}
\pi\mapsto\Hom_{\J^1_{{\rm max}}}(\kmax,\pi)
\end{equation*}
from $\Rr(\G)$ to $\Rr(\jmax^{}/\J^1_{{\rm max}})$ and set 
$\Gg=\jmax^{}/\J^1_{{\rm max}}$; this is the functor denoted~$\KM_{\kmax}$ in~\S\ref{SST}.

Let $\M$ be a standard Levi subgroup of $\G$, associated with a 
compo\-si\-tion $\a=(m_1,\dots,m_r)$ of $m$.
We assume that it is $\b$-\textit{admissible}, that is, the $\F$-algebra $\F[\b]$, 
denoted $\E$,
can be embedded in $\A_{m_i}$ for all $i$.
Equivalently, $m_id$ is a multiple of the degree of $\E$ over $\F$ 
for all $i$.
Let $\P$ be the corresponding standard parabolic subgroups of $\G$, 
and write $\N$ for its unipotent radical.

We fix an isomorphism of $\E$-algebras $\Phi$ between $\B$ and 
$\Mat_{m'}(\D')$ that identifies $\AA(\Lamax)\cap\B$ with the maximal standard 
order made of matrices with integer entries.
We choose an $\E$-pure lattice sequence $\La$ such that:
\begin{equation}
\label{YSL}
\U(\La)\cap\mult\B = (\U_1(\Lamax)\cap\mult\B)(\P\cap\U(\Lamax)\cap\mult\B).
\end{equation}
The image $\Phi(\U(\La)\cap\mult\B)$
is the standard parahoric subgroup of $\GL_{m'}(\D')$ whose 
reduction mod $\p_{\D'}$ is made of upper block triangular matrices of sizes 
$(m'_1,\dots,m'_r)$, with:
\[
m'_id'=\frac{m_id}{[\E:\F]},
\quad
i\in\{1,\dots,r\},
\]
where $d'$ is the reduced degree of $\D'$ over $\E$. 
Moreover, $\La$ can be chosen such that it satisfies the conditions of the 
following lemma. 

\begin{lemm}
\label{L1}
There is an $\E$-pure lattice sequence $\La$ on $\D^m$ satisfying 
\eqref{YSL} and such that: 
\begin{eqnarray*}
\U(\La)&\subseteq&\U(\Lamax);\\
\U_1(\La)\cap\N^-&=&\U_1(\Lamax)\cap\N^-.
\end{eqnarray*}
\end{lemm}

\begin{proof}
We fix a simple left $\E\otimes_\F\D$-module $\V_0$, and form the 
simple left $\B$-module
\begin{equation*}
\V_\B=\Hom_{\E\otimes_\F\D}(\V_0,\D^m).
\end{equation*}
The $\E$-algebra opposite to $\End_{\B}(\V_\B)$ is isomorphic to $\D'$. 
Write $\A_0=\End_{\D}(\V_0)$ and $\AA_0$ for the unique hereditary order 
in $\A_0$ normalized by $\mult\E$, and $\PP_0$ for its Jacobson radical. 
If we identify $\A$ with $\Mat_{m'}(\A_0)$, then $\AA(\Lamax)$ identifies with 
$\Mat_{m'}(\AA_0)$. 
Then choose $\La$ such that:
\begin{equation*}
\AA(\La)= 
\begin{pmatrix}
\AA_0 & \cdots & \AA_0 \\
\vdots & \ddots & \vdots \\
\PP_0 & \cdots & \AA_0 \\
\end{pmatrix}
\subseteq 
\begin{pmatrix}
\AA_0 & \cdots & \AA_0 \\
\vdots & \ddots & \vdots \\
\AA_0 & \cdots & \AA_0 \\
\end{pmatrix}
=\AA(\Lamax)
\end{equation*}
(see~\cite{VS3}). 
We have:
\begin{equation*}
\aa_1(\La)= 
\begin{pmatrix}
\PP_0 & \cdots & \AA_0 \\
\vdots & \ddots & \vdots \\
\PP_0 & \cdots & \PP_0 \\
\end{pmatrix}
\supseteq 
\begin{pmatrix}
\PP_0 & \cdots & \PP_0 \\
\vdots & \ddots & \vdots \\
\PP_0 & \cdots & \PP_0 \\
\end{pmatrix}
=\aa_1(\Lamax).
\end{equation*}
Therefore both $\aa_1(\La)\cap\nn^-$ and $\aa_1(\Lamax)\cap\nn^-$ are made of blocks 
with values in $\PP_0$.
\end{proof}

Write $\t$ for the transfer of $\tmax$ to $\Cc(\La,0,\b)$ in the sense 
of~\cite{SeSt1}, and $\k$ for the unique $\b$-ex\-tension of $\t$ such that:
\begin{equation}
\label{COHERENCE}
\Ind^{(\U(\La)\cap\mult\B)\U_1(\La)}_{\J}(\k)\simeq
\Ind^{(\U(\La)\cap\mult\B)\U_1(\La)}_{(\U(\La)\cap\mult\B)\jmax^1}(\kmax)
% (\kmax|_{(\U(\La)\cap\mult\B)\jmax^1})
\end{equation}
where $\J=\J(\b,\La)$. 
We also write $\J_\P=\H^1(\J\cap\P)$ and $\k_\P$ for the unique 
irreducible representation of $\J_\P$ that is trivial on 
$\J_\P\cap\N$ and $\J_\P\cap\N^-$ and such that, 
if we restrict $\k_\P$ to $\J\cap\M$, we get: 
\begin{equation*}
\J\cap\M=\J_1\times\dots\times\J_r,
\quad
\k_\P|_{\J\cap\M}=\k_1\otimes\dots\otimes\k_r,
\end{equation*}
where $\J_i=\J(\b,\La_i)$ and $\k_i$ is a $\b$-extension with respect to some 
simple stratum $[\La_i,n_i,0,\b]$ in $\A_{m_i}$. 
We have an isomorphism of representations of $\J$:
\begin{equation}\label{eqn:kP}
\Ind^\J_{\J_\P}(\k_\P)\simeq\k.
\end{equation}
We write $\J_{{\rm max},\a}^{}=\J\cap\M$,
$\J_{{\rm max},\a}^{1}=\J^1\cap\M$ and 
$\k_{{\rm max},\a}=\k_\P|_{\J\cap\M}$. 
We have a functor: 
\begin{equation*}
\KM_{\M}:\pi\mapsto\Hom_{\J^1_{{\rm max},\a}}(\k_{{\rm max},\a},\pi)
\end{equation*}
from $\Rr(\M)$ to $\Rr(\J_{{\rm max},\a}^{}/\J_{{\rm max},\a}^{1})$. 

The groups $\J\cap\M/\J^1\cap\M$, 
$(\U(\La)\cap\mult\B)\J^1_{{\rm max}}/(\U_1(\La)\cap\mult\B)\J^1_{{\rm max}}$ and 
$\J_{{\rm max},\a}^{}/\J_{{\rm max},\a}^{1}$ will all be identified,
and all of them will be denoted $\Mm$. 
For simplicity, we will write:
\begin{eqnarray*}
\U&=&(\U(\La)\cap\mult\B)\U_1(\La),\\
\U^1&=&\U_1(\La)\cap\U=\U_1(\La),\\
\SS&=&(\U(\La)\cap\mult\B)\J^1_{{\rm max}},\\
\SS^1&=&\U_1(\La)\cap\SS=(\U_1(\La)\cap\mult\B)\J^1_{{\rm max}}.
\end{eqnarray*}

\subsection{}

We write $\KB$ for the functor:
\begin{equation*}
\pi\mapsto\Hom_{\SS^1}(\kmax|_{\SS},\pi)
\end{equation*}
from $\Rr(\SS)$ to $\Rr(\Mm)$. 
Note that this fits in the framework of~\S\ref{SST}, with:
\begin{equation*}
\textbf{\textsf{J}}=\SS,
\quad
\textbf{\textsf{J}}^1=\SS^1,
\quad
\textbf{\textsf{H}}^1=(\U_1(\La)\cap\mult\B)\H^1_{{\rm max}}
\quad
\bk=\kmax|_{\SS},
\end{equation*}
since, by the construction of~$\b$-extensions in~\cite{VS2}:
\begin{enumerate}
\item 
the restriction of~$\kmax$ to~$\SS^1$ is the unique (irreducible) 
representation~$\tilde\eta$ which extends~$\eta_{{\rm max}}$ and such 
that~$\Ind_{\SS^1}^{\U^1}(\tilde\eta)$ is equivalent 
to~$\Ind_{\J^1}^{\U^1}(\eta)$; 
\item 
the restriction of~$\tilde\eta$ to~$(\U_1(\La)\cap\mult\B)\H^1_{{\rm max}}$ is a 
multiple of the character~$\tilde\theta$ given by: 
\[
\tilde\theta(uh)=\theta(u)\theta_{{\rm max}}(h), 
\]
for~$u\in\U_1(\La)\cap\mult\B$ and~$h\in\H^1_{{\rm max}}$.  
(Note that this is 
well-defined, by~\cite[Th\'eor\`eme~2.13]{SeSt1}.)  
\end{enumerate}

\begin{prop}
\label{F1}
For any smooth representation~$\pi$ of~$\M$, we have
\[
\KM_\M(\pi) \simeq \KB\(\Ind^{\P\SS}_{\P}(\pi)\)
\] 
as representations of~$\Mm$.
\end{prop}

\begin{proof}
Let $\pi$ be a smooth representation of $\M$. Then, by inflation, we have
\[
\KM_{\M}(\pi)=\Hom_{\J^1_{{\rm max},\a}}(\k_{{\rm max},\a},\pi)\simeq\Hom_{\J^1\cap\P}(\k_\P,\pi).
\]
By Frobenius reciprocity and the Mackey formula, this is isomorphic to
\[
\Hom_{\J_\P^1}(\k_\P,\Ind_{\J\cap\P}^{\J_\P}(\pi)).
\]
Again we are in the situation of~\S\ref{SST}, 
with~$\textbf{\textsf{J}}=\J_\P$,~$\textbf{\textsf{J}}^1=\J^1_\P$,~$\bk=\k_\P$, 
and~$\bt=\t_\P$, the character of Corollary~\ref{coro:tP}.  Thus, using the 
notation of~\S\ref{SST} and Lemma~\ref{L0}, we get 
\begin{equation}\label{eqn:kpjp}
\KM_{\M}(\pi)\simeq \KM_{\k_\P}\circ\Ind_{\P}^{\P\J_\P}(\pi).
\end{equation}
We decompose~$\P\U$ as a disjoint union of double cosets~$\P u\J_\P$, where 
the double coset representatives~$u$ may, and will, be chosen 
in~$\U\cap\N^-=\U_1(\La)\cap\N^-$; then~$\Ind_\P^{\P\U}(\pi)=\bigoplus_u 
\Ind_\P^{\P u\J_\P}(\pi)$. 

By Corollary~\ref{cor:Annonce}, we have 
that~$\KM_{\k_\P}\circ\Ind_{\P}^{\P u\J_\P}$ is non-zero if and only 
if~$\t_\P$ is trivial on~$\H_\P^1\cap\N^u$, which, by Corollary~\ref{coro:tP}, 
implies~$u\in\J_\P^1$. 
Thus~\eqref{eqn:kpjp} implies 
\[
\KM_{\M}(\pi)\simeq \KM_{\k_\P}\circ\Ind_{\P}^{\P\U} \pi\simeq \Hom_{\J_\P^1}(\k_\P,\Ind_{\P\cap\U}^\U(\pi)).
\]
Write~$\rho$ for the irreducible induced
representation~$\Ind^\U_{\J_\P}(\k_\P)$ which, by~\eqref{COHERENCE} 
and~\eqref{eqn:kP}, is isomor\-phic to~$\Ind^{\U}_{\SS}(\kmax|_{\SS})$.  
Then, again by Frobenius and Mackey, we get 
\[
\KM_{\M}(\pi)\simeq\Hom_{\U^1}(\rho,\Ind_{\P\cap\U}^\U(\pi))\simeq 
\Hom_{\SS^1}(\kmax|_{\SS},\Ind_{\P\cap\U}^\U(\pi))\simeq \KB\circ\Ind_{\P}^{\P\U} (\pi),
\]
applying Lemma~\ref{L0} again.

As before, we decompose~$\P\U$ as a disjoint union of double cosets~$\P u\SS$, 
where the double coset representatives~$u$ lie in~$\U\cap\N^-$ which, by 
Lemma~\ref{L1}, is~$\U_1(\La)\cap\N^-$; then~$\Ind_\P^{\P\U}\pi=\bigoplus_u 
\Ind_\P^{\P u\SS}\pi$.  
Now Corollary~\ref{cor:Annonce} shows that the 
functor~$\KB\circ\Ind_{\P}^{\P u\SS}$ is nonzero on~$\Rr(\M)$ if and only 
if~$\tilde\theta$ is trivial on~$(\U_1(\La)\cap\mult\B)\H^1_{{\rm max}}\cap\N^u$; in 
particular, restricting to~$\H^1_{{\rm max}}$ and applying Lemma~\ref{charN1}, 
we see that~$u\in\P\SS$ so 
\[
\KM_{\M}(\pi)\simeq\KB\circ\Ind_{\P}^{\P\U} (\pi) = \KB\circ\Ind_{\P}^{\P\SS}(\pi).
\]
This ends the proof of Proposition~\ref{F1}.
\end{proof}

The following lemma relates the functor~$\KB$ back to our functor~$\KM$. 
We put~$\Pp=\SS/\jmax^1$, which is a parabolic sub\-group 
of~$\Gg=\jmax^{}/\jmax^1$ with Levi component~$\Mm$.  
We regard representations of~$\Mm$ as representations of~$\Pp$ by inflation. 

\begin{lemm}
\label{IndInd}
For any smooth representation~$\pi$ of~$\M$, we have
\begin{equation*}
\KB\(\Ind_{\P}^{\P\SS}(\pi)\) \simeq \KM\(\Ind_{\P}^{\P\SS}(\pi)\) 
\end{equation*}
as representations of~$\Pp$.
\end{lemm}

\begin{proof}
We clearly have an inclusion of 
spaces~$\Hom_{\SS^1}(\kmax,\Ind_{\P}^{\P\SS}\pi)\subseteq\Hom_{\jmax^1}(\kmax,\Ind_{\P}^{\P\SS}\pi)$ 
and, if we check that we have equality here, it is then straightforward that 
the actions of~$\Pp$ are the same. Write $\EuScript{V}$ for the space of 
$\kmax$.  

The action of~$\U_1(\La)\cap\mult\B$ on~$\EuScript{V}$ is a multiple 
of~$\tilde\theta|_{\U_1(\La)\cap\mult\B}$, which factors through the reduced 
norm. 
Thus, for~$u\in \U_1(\La)\cap\mult\B\cap\N$, we have~$\kmax(u)=\id_{\EuScript{V}}$.  
Now let:
\begin{equation*}
f\in\Hom_{\jmax^1}(\kmax,\Ind_{\P}^{\P\SS}\pi)
\end{equation*} 
and~$v\in\EuScript{V}$, and put~$\varphi=f(v)$.  
For~$j\in\jmax^1$ and~$u\in \U_1(\La)\cap\mult\B\cap\N$, 
we have~$\eta_{{\rm max}}(u^{-1}ju)=\eta_{{\rm max}}(j)$ and~$\pi(u)$ 
acts trivially on the space of~$\pi$ so 
\[
(u\cdot\varphi)(j)=\varphi(ju)=\varphi(u^{-1}ju)=f(\eta_{{\rm max}}(u^{-1}ju)v)(1)=f(\eta_{{\rm max}}(j)v)(1)=\varphi(j).
\]
Since~$\P\SS=\P\jmax^1$, this implies that~$u\cdot\varphi=\varphi$.  
Thus 
\[
f(\kmax(u)v)=f(v)=u\cdot f(v)
\]
and~$f\in\Hom_{\SS^1}(\kmax,\Ind_{\P}^{\P\SS}\pi)$ since~$\SS^1=(\U_1(\La)\cap\mult\B\cap\N)\jmax^1$.
\end{proof}

\subsection{}

Then next step is to relate parabolic induction in the finite reductive group~$\Gg$ to induction inside~$\jmax^{}$.

\begin{lemm}
\label{indS}
For any smooth representation $\tau$ of $\SS$, we have:
\begin{equation*}
\KM\(\Ind^{\J_{{\rm max}}}_\SS(\tau)\)\simeq
\ia_{\Pp}^{\Gg}\(\KM(\tau)\)
\end{equation*}
as representations of $\Gg$. 
\end{lemm}

Note that~$\KM(\tau)=\Hom_{\jmax^1}(\kmax,\tau)$ is viewed here as a representation of~$\Pp$ by restriction.

\begin{proof}
As above, write $\EuScript{V}$ for the space of $\kmax$. 
Given $f\in\KM(\Ind^{\J_{{\rm max}}}_\SS(\tau))$,
we define a function $\fb$ by:
\begin{equation*}
\fb(\dot x):v\mapsto f(x^{-1}\cdot v)(x)
\end{equation*}
for all $x\in\jmax^{}$ and $v\in\EuScript{V}$, where $\dot x$ is the class of 
$x$ in $\Gg$. 

We first need to check that $\fb$ is well defined. 
Let $z\in\jmax^{1}$.
For $v\in\EuScript{V}$ and $x\in\jmax^{}$, we have:
\begin{eqnarray*}
f(z^{-1}x^{-1}\cdot v)(xz)
&=&[z^{-1}\cdot f(x^{-1}\cdot v)](xz)\\
&=&f(x^{-1}\cdot v)(xz\cdot z^{-1})\\
&=&f(x^{-1}\cdot v)(x).
\end{eqnarray*}
We now check that $\fb$ takes its values in 
$\ia_{\Pp}^{\Gg}(\Hom_{\jmax^1}(\kmax,\tau))$. 
Given $v\in\EuScript{V}$, $x\in\jmax^{}$ and $j\in\jmax^1$, we first have: 
\begin{eqnarray*}
\fb(\dot x)(j\cdot v)
&=&f(x^{-1}j\cdot v)(x)\\
&=&f(x^{-1}j\cdot v)(j\cdot j^{-1}x)\\
&=&\tau(j)[f(x^{-1}j\cdot v)(j^{-1}x)]
\end{eqnarray*}
which is equal to $\tau(j)[\fb(\dot x)(v)]$
since $j^{-1}x$ and $x$ have the same image in $\Gg$. 
Now given $s\in\SS$, $x\in\jmax^{}$ and $v\in\EuScript{V}$, 
we have:
\begin{eqnarray*}
\fb(\dot s\dot x)(v)
&=&f(x^{-1}s^{-1}\cdot v)(sx)\\
&=&\tau(s)[f(x^{-1}s^{-1}\cdot v)(x)].
\end{eqnarray*}
On the other hand, we have:
\begin{eqnarray*}
[\dot s\cdot\fb(\dot x)](v)
&=&[\tau(s)\circ\fb(\dot x)\circ\kmax(s)^{-1}](v)\\
&=&\tau(s)[\fb(\dot x)(s^{-1}\cdot v)]
\end{eqnarray*}
and this coincides with $\fb(\dot s\dot x)(v)$. 

We now check that $f\mapsto\fb$ is a $\Gg$-homomorphism.  
Given $x,y\in\jmax^{}$ and $v\in\EuScript{V}$, we have:
\begin{eqnarray*}
\overline{\dot y\cdot f}(\dot x)(v)
&=&[\Ind^{\Gg}_{\Pp}(\tau)(y)\circ f\circ\kmax(y)^{-1}](x^{-1}\cdot v)(x)\\
&=&f(y^{-1}x^{-1}\cdot v)(xy)
\end{eqnarray*}
which is equal to $\fb(\dot x\dot y)(v)$ and 
gives us $\overline{\dot y\cdot f}(\dot x)=\fb(\dot x\dot y)$, 
thus the expected relation $\overline{\dot y\cdot f}=\dot y\cdot\fb$.

The map $f\mapsto\fb$ is clearly injective. 
Now let $\phi$ be some function in 
$\ia_{\Pp}^{\Gg}(\Hom_{\SS^1}(\kmax|_{\SS},\tau))$.
We define a function $f$ from $\EuScript{V}$ to 
$\Ind^{\J_{{\rm max}}}_\SS(\tau)$ by:
\begin{equation*}
f(v)(x)=\phi(\dot x)(x\cdot v).
\end{equation*}
Checking that $f\in\KM(\Ind^{\J_{{\rm max}}}_\SS(\tau))$
and that $\fb=\phi$ is similar to the calculations above, and this completes the proof of the lemma.
\end{proof}

Putting this together with the results of the previous subsection, we get:

\begin{coro}
\label{zonzon}
For any smooth representation~$\pi$ of~$\M$, we have
\begin{equation*}
\KM\(\Ind^{\P\jmax^{}}_\P(\pi)\)\simeq\ia_{\Pp}^{\Gg}\(\KM_{\M}(\pi)\)
\end{equation*}
as representations of~$\Gg$.
\end{coro}

\begin{proof}
Putting together Proposition~\ref{F1} with Lemmas~\ref{IndInd},~\ref{indS}, we get
\begin{equation*}
\ia_{\Pp}^{\Gg}\(\KM_\M(\pi)\)  
\simeq\KM\(\Ind^{\jmax^{}}_{\SS}\(\Ind^{\P\SS}_{\P}(\pi)\)\),
\end{equation*}
while $\Ind^{\jmax^{}}_{\SS}\(\Ind^{\P\SS}_{\P}(\pi)\)\simeq\Ind^{\P\jmax^{}}_{\P}(\pi)$, from Lemma~\ref{L0} and the fact that~$\P\cap\SS=\P\cap\jmax^{}$. 
\end{proof}

\begin{prop}
\label{MainTheoK}
For any smooth representation~$\pi$ of~$\M$, we have an isomorphism
\begin{equation*}
\KM\(\ia_\P^\G(\pi)\)\simeq\ia_{\Pp}^{\Gg}\(\KM_\M(\pi)\)
\end{equation*}
as representations of~$\Gg$.
\end{prop}

\begin{proof}
The result is true in the case where $\R$ is the field of complex numbers: 
the method used by Schnei\-der and Zink in~\cite{SZ}, 
based on equivalences of categories given by the theory of types 
for complex representations, applies \emph{mutatis mutandis} 
to inner forms of $\GL_n$, for $n\>1$.  
To\-ge\-ther  with  Corollary~\ref{zonzon},  we  get that,  for~$\pi$  any  smooth
complex representation of~$\M$, the canonical inclusion:
\begin{equation*}
\KM(\Ind^{\P\jmax^{}}_\P(\pi)) \subseteq \KM(\ia_\P^\G(\pi))
\end{equation*} 
is an equality.  
In particular, since the right hand side 
% $\KM\(\ia_\P^\G(\pi)\)$ 
is finite-dimensional for any 
irreducible (so admissible) complex representation~$\pi$ of~$\M$, the functor 
$\KM\circ\ia_\P^{\P g\jmax^{}}$ is zero on~$\Irr(\M)$, for 
any~$g\not\in\P\jmax^{}$. 
By Corollary~\ref{cor:Annonce}, this implies that, for~$g\in\G$,
\begin{equation}
\label{Leila2}
\tmax^{} \text{ is trivial on } \H^1_{{\rm max}}\cap\N^{g}
\quad\Leftrightarrow\quad
g\in\P\jmax^{}
\end{equation}
for any complex maximal simple character $\tmax$. 
As $\H^1_{{\rm max}}$ is a pro-$p$-group, 
\eqref{Leila2} holds also for any \textit{modular} maximal simple character. 
Thus, by Corollary~\ref{cor:Annonce} again, the equality
\begin{equation*}
%\label{Leila}
\KM\(\ia_\P^\G(\pi)\)=\KM\(\Ind^{\P\jmax^{}}_\P(\pi)\)
\end{equation*}
holds for all smooth $\R$-representations $\pi$ of $\M$.  
The result follows from Corollary~\ref{zonzon}.
\end{proof}

\begin{rema} We have proved that the
functors~$\KM\circ\Ind_\P^\G$ and~$\Ind_\Pp^\Gg\circ\KM_{\M}$
from~$\Rr(\M)$ to~$\Rr(\G)$ behave in the same way on objects. It
seems likely that similar proofs would show that they behave in the
same way on morphisms so that the two functors are in fact
isomorphic.
\end{rema}

%%%%%%%%%%%%%%%%%%%%%%%%%%%%%%%%%%%%%%%%%%%%%%%%%%%%%%%%%%%%%%%%%%%%%%%%
\section{Semisimple supertypes}
\label{JJM}
%%%%%%%%%%%%%%%%%%%%%%%%%%%%%%%%%%%%%%%%%%%%%%%%%%%%%%%%%%%%%%%%%%%%%%%%

In this section, we first recall briefly the basic properties of, and data 
attached to, semisimple supertypes, for which we refer to \cite{SeSt2,MS1} 
for more details, and we explain the functor~$\KM$ in this situation. The main 
result is Theorem~\ref{Cagliostro}, which extends to the semisimple case 
the main result of the previous section: the functor~$\KM$ commutes with 
parabolic induction.

\subsection{}
\label{ss.JJM1}

Let $\a=(m_1,\dots,m_r)$ be a composition of $m$. 
For all $i\in\{1,\dots,r\}$, let $(\J_i,\l_i)$ be a maximal simple type attached 
to a simple stratum $[\La_i,n_i,0,\b_i]$ in $\A_{m_i}$. 
We write $\M$ for the standard Levi 
subgroup $\G_{m_1}\times\dots\times\G_{m_r}$ in $\G$ and:
\begin{equation*}
\J_\a=\J_1\times\dots\times\J_r,
\quad
\l_\a=\l_1\otimes\dots\otimes\l_r. 
\end{equation*}
A pair of the form $(\J_\a,\l_\a)$ is called a maximal simple type of $\M$.  
Associated to it, there is a pair $(\BJ,\bl)$ called a semisimple type of $\G$ 
(see \cite{SeSt2,MS1}). 
For any parabolic subgroup $\P$ of $\G$ with Levi component $\M$, 
the pair $(\BJ,\bl)$ satisfies the following properties:
\begin{enumerate}
\item
the kernel of $\bl$ contains $\BJ\cap\N$ and $\BJ\cap\N^-$, 
where $\N$ and $\N^-$ denote the unipotent radicals of $\P$ and $\P^-$, 
the parabolic subgroup opposite to $\P$ with respect to $\M$;
\item
one has $\BJ\cap\M=\J_\a$ and $\bl|_{\BJ\cap\M}=\l_\a$;
\end{enumerate}
(these two conditions say that $(\BJ,\bl)$ is \textit{decomposed} above the 
pair $(\J_\a,\l_\a)$ with respect to $(\M,\P)$ in the sense of 
\cite[Definition 6.1]{BK1}), 
plus another technical condition saying that the pair $(\BJ,\bl)$ is a cover 
of $(\J_\a,\l_\a)$ in the sense of \cite[Definition 8.1]{BK1}. Note
that there is considerable flexibility in the construction of
semisimple types; in particular, there is a
(not entirely arbitrary) choice of lattice sequence~$\La$ on $\D^m$ such
that:
\begin{equation*}
\U(\La)\cap\M=\U(\La_1)\times\cdots\times\U(\La_r)
\end{equation*}
(see~\cite[\S7.1]{SeSt2} and \cite[\S2.8-9]{MS1} for the precise condition). 
In particular, we may and will assume that the lattice sequences 
$\La_1,\dots,\La_r$ and $\La$ all have the same period. 

Given $\pi_i$ a representation of $\G_{m_i}$ for all $i\in\{1,\dots,r\}$, 
we write $\pi_1\tdt\pi_r$ for the represen\-tation 
$\Ind^\G_\P(\pi_1\odo\pi_r)$, where $\P$ is the parabolic subgroup of $\G$ 
with Levi component $\M$ made of upper triangular matrices. 

An important relationship between $(\BJ,\bl)$ and 
$(\J_1,\l_1),\dots,(\J_r,\l_r)$ is that there is an isomorphism of representations of $\G$:
\begin{equation*}
\ind^\G_\BJ(\bl)\simeq\ind^{\G_{m_1}}_{\J_1}(\l_1)\tdt\ind^{\G_{m_r}}_{\J_1}(\l_r)
\end{equation*}
(see \cite{Blondel}). 
Note, in particular, that this is independent of any choices made in the 
construction of~$(\BJ,\bl)$. 

\begin{defi}
\begin{enumerate}
\item
When $(\J_1,\l_1),\dots,(\J_r,\l_r)$ are maximal simple supertypes, 
$(\BJ,\bl)$ is called a~\emph{semisimple supertype} of~$\G$.
\item The~\emph{equivalence class} of a semisimple type~$(\BJ,\bl)$ is the 
  set~$[\BJ,\bl]$ of all semi\-simple supertypes $(\BJ',\bl')$ of $\G$ such 
  that $\ind^{\G}_{\BJ'}(\bl')$ is isomorphic to $\ind^\G_\BJ(\bl)$. 
\end{enumerate}
\end{defi}

Together with $\BJ$, we also have a normal open subgroup
$\BJ^1$ and an 
irreducible representation $\bn$ of $\BJ^1$ (see \cite[\S2.10]{MS1}).
When restricting $\bl$ to $\BJ^1$, we get a direct sum of copies of $\bn$.
There is a decomposition of the form:
\begin{equation}
\label{DECSST}
\bl\simeq\bk\otimes\bs,
\end{equation}
where $\bk$ is an 
irreducible representation of $\BJ$ extending $\bn$ and $\bs$ is an 
irreducible representation of $\BJ$ trivial on $\BJ^1$.  
The re\-pre\-sentation~$\bk$ has the property that its intertwining is the
same as that of~$\bn$, but is not uniquely determined by this
condition; thus there is a choice of~$\bk$ to be made in the
decomposition~\eqref{DECSST}.

For each~$i\in\{1,\ldots,r\}$, we have a maximal simple character 
$\t_i$ attached to the simple stratum $[\La_i,n_i,0,\b_i]$, an
iso\-mor\-phism of $\F[\b_i]$-algebras~$\B_i\simeq\Mat_{m'_i}(\D'_i)$ 
for a suitable $\F[\b_i]$-di\-vi\-sion algebra $\D'_i$, 
and isomorphisms of groups:
\begin{equation*}
\BJ/\BJ^1\simeq \J^{}_1/\J^1_1\times\dots\times\J^{}_r/\J^1_r
\simeq\GL_{m'_1}(\kk_{\D'_1})\times\dots\times\GL_{m'_r}(\kk_{\D'_r});
\end{equation*}
we denote by $\Mm$ this latter group. 
The representation $\bk$ is trivial on $\BJ\cap\N$ and $\BJ\cap\N^-$, 
and its restriction to $\BJ\cap\M=\J_\a$ is of the form 
$\k_\a=\k_1\otimes\dots\otimes\k_r$, where $\k_i$ is a maximal 
$\b_i$-extension of $\t_i$. 

For each $i$, there is a decomposition $\l_i=\k_i\otimes\s_i$, 
where $\s_i$ is an irreducible representation of $\J^{}_i$ trivial on $\J^1_i$ 
that identifies with a cuspidal representation of $\GL_{m'_i}(\kk_{\D'_i})$, 
and $\bs$ identifies with the irreducible cuspidal representation 
$\s_1\otimes\dots\otimes\s_r$ of $\Mm$.  

\subsection{}
\label{ss:hom}

We will need to recall some more detail of the structure of semisimple
supertypes~$(\BJ,\bl)$, which we begin in this section. 

We write $\TT_i$ for the endo-class of $\t_i$ (see \cite{BSS} for
the definition of endo-class) and assume first that the
endo-classes~$\TT_i$ all coincide, the
so-called \emph{homogeneous case}. In this case, we may and will
assume that the elements~$\b_1,\dots,\b_r$
% $\b_i$ 
are all equal to (the image of) a
single element~$\b$ and that the characters~$\t_i$ are related by the
transfer maps (in other words, they are realizations of the same
ps-character -- see~\cite{BSS}). We put~$\E=\F[\b]$ and denote by~$\B$
the centralizer of~$\E$ in~$\A$, so that~$\B\simeq\M_{m'}(\D')$,
where~$\D'$ is a suitable~$\E$-division algebra. Similarly, we
write~$\B_i\simeq\M_{m'_i}(\D')$ for the centralizer of~$\E$ in~$\A_{m_i}$. 

We choose a simple stratum $[\Lamax,n_{{\rm max}},0,\b]$ in $\A$ and 
an isomorphism of~$\E$-algebras $\Phi$ from $\B$ to $\M_{m'}(\D')$ 
with the following properties: 
\begin{enumerate}
\item 
$\U(\Lamax)\cap\mult\B$ is a maximal compact subgroup of $\mult\B$ 
that contains $\U(\La)\cap\mult\B$;
\item
$\Phi(\U(\Lamax)\cap\mult\B)$ and $\Phi(\U(\La)\cap\mult\B)$ are both 
standard parahoric subgroups of $\GL_{m'}(\D')$;
% \item
% for each $i\in\{1,\dots,r\}$, $\Phi(\U(\La)\cap\B_i^\times)$ is the 
% standard maximal compact open subgroup. 
\end{enumerate}
% We write~$\BBmax$ for the
% standard maximal order in~$\B$ and~$\BB$ for the standard hereditary
% order associated with the composition~$\a'=(m'_1,\ldots,m'_r)$ of~$m'$,
% as in~\S\ref{badm}. We take~$\AAmax$ the hereditary order in~$\A$
% such that~$\AAmax\cap\B=\BBmax$, and 
By passing to an equivalent type if necessary, 
we will assume that $\U(\La)\subseteq\U(\Lamax)$ as in Lem\-ma~\ref{L1}.  

% We choose a simple stratum $[\La',n',0,\b]$ in $\A$ such that 
% $\U(\La')\cap\B=\U(\La)\cap\B$ as in Lemma~\ref{L1}. 
% The lattice sequence~$\La'$ satisfies the conditions of~\cite[\S7.1]{SeSt2} 
% for the construction of a semisimple type so, by passing to an equivalent type
% if necessary, we will assume that
% this is the one used in the construction of~$(\BJ,\bl)$. 

We are now in the situation of~\S\ref{badm}, with~$\t$ the transfer of~$\t_i$ 
to~$\Cc(\La,0,\b)$ (which is independent of~$i$), and we take the notation from there. 
We have~$\BJ=\J_\P$ and~$\bk=\k_\P$ 
for some choice of~$\b$-ex\-tension~$\kmax$ of~$\tmax$; 
it is thus this choice of~$\kmax$ which imposes the choice of~$\bk$ in~\S\ref{ss.JJM1}. 
The group~$\Mm$ is a Levi subgroup of:
\[
\Gg=\GL_{m'}(\kk_{\D'})\simeq \jmax^{}/\jmax^1
\]
so we get a supercuspidal pair~$(\Mm,\bs)$ of~$\Gg$,
where~$\bs=\s_1\otimes\dots\otimes\s_r$
% $\bs=\bigotimes_{i=1}^r\s_i$ 
is as above. 
Taking~$\Ga$ to be 
the group~$\Gal(\kk_{\D'}/\kk_{\E})$, we also get an equivalence 
class~$[\Mm,\bs]$ of supercuspidal pairs, in the sense of 
Definition~\ref{DefELS}.

The group~$\Gg$ and the conjugacy class of~$\Mm\subseteq\Gg$ are
uniquely determined by the semisimple type~$(\BJ,\bl)$, 
independently of the decomposition~$\bl=\bk\otimes\bs$. 
The representation~$\bk$ is
not uni\-que\-ly determined but, once it is fixed (or, equivalently, the
representation~$\kmax$ is fixed), it determines the equivalence
class~$[\Mm,\bs]$, as well as the functor:
\[
\KM=\KM_{\kmax}:\Rr(\G)\to \Rr(\Gg).
\]
Moreover, every equivalence class~$[\Mm',\bs']$ arises from some
homogeneous semisimple supertype: $\Mm'$ determines a composition~$\a'$
of~$m'$ and hence 
% a standard hereditary order~$\BB'$ in~$\BBmax$, and also 
a Levi subgroup~$\M'$ of~$\G$ with standard para\-bo\-lic subgroup~$\P'$;
then we may make the constructions of~\S\ref{badm} to get a
pair~$(\BJ',\bl')$, with~$\BJ'=\J_{\P'}$
and~$\bl'=\k_{\P'}\otimes\bs'$, which is a homogeneous semisimple supertype
with the required property.

\subsection{}

Now we consider the general case, when the endo-classes~$\TT_i$ may differ.
Let $\TT=\TT(\BJ,\bl)$ be the formal sum:
\begin{equation*}
\sum\limits_{i=1}^{r}\frac{m_id}{[\F[\b_i]:\F]}\cdot\TT_i
\end{equation*}
in the semigroup of finitely supported maps 
$\{\text{endo-classes over }\F\}\to\NN$
(with $\NN$ the semigroup of nonnegative integers).
The fibers of the map $i\mapsto\TT_i$ define a partition:
\begin{equation*}
\{1,\dots,r\}=\I_1\cup\dots\cup\I_{\ll}
\end{equation*}
for some $s\>1$. 
Renumbering, we may assume that the $\I_j$ (for $j\in\{1,\dots,\ll\}$) are of the form:
\begin{equation*}
\I_j=\{i\in\{1,\dots,r\}\ |\ a_{j-1}<i\<a_j\}
\end{equation*}
for some integers $0=a_0<a_1<\dots<a_{\ll}=r$. 
For all $j\in\{1,\dots,\ll\}$, we write:
\begin{equation*}
n_j=\sum_{i\in\I_j} m_i,
\quad
\M_j=\prod_{i\in\I_j}\G_{m_i},
\end{equation*}
and $\P_j$ the standard parabolic subgroup of~$\G_{n_j}$ with Levi 
subgroup~$\M_j$. 
Let~$\L$ be the standard Levi subgroup~$\G_{n_1}\times\cdots\times\G_{n_\ll}$ 
in~$\G$; thus we have ~$\P\cap\L=\P_1\times\cdots\times\P_\ll$.
From the construction of semisimple types, and by passing to an equivalent 
semisimple type as before if necessary, we have:
\[
\BJ\cap\L= \BJ_1\times\cdots\times\BJ_\ll,
\quad
\bl_{\BJ\cap\L}=\bl_1\otimes\cdots\otimes\bl_\ll,
\] 
where each~$(\BJ_j,\bl_j)$ is a homogeneous semisimple supertype, as described in
the previous section. 
In particular, for each~$j\in\{1,\dots,\ll\}$, we choose a
pair~$(\J_{{\rm max},j},\k_{{\rm max},j})$ and have the group~$\Gg_j$ and
the supercuspidal equivalence class~$[\Ll_j,\bs_j]$. The choice of the 
representations $\k_{{\rm max},j}$ imposes the choice of~$\bk$ 
in~\S\ref{ss.JJM1} (and vice versa). 

Now write~$\mu=(n_1,\ldots,n_\ll)$ and:
\begin{equation*}
\J_{{\rm max},\mu}=\J_{{\rm max},1}\tdt\J_{{\rm max},\ll},
\quad
\k_{{\rm max},\mu}=\k_{{\rm max},1}\odo\k_{{\rm max},\ll},
\end{equation*}
so that: 
\[
\J^{}_{{\rm max},\mu}/\J^1_{{\rm max},\mu}\simeq\Gg_1\times\cdots\times\Gg_\ll;
\]
we denote the latter group by~$\Gg$. 
We also get an
isomorphism of groups~$\Mm\simeq\Mm_1\times\cdots\times\Mm_\ll$
which iden\-tifies~$\bs$ with~$\bs_1\otimes\cdots\otimes\bs_\ll$. 
Then~$(\Mm,\bs)$ is a
supercuspidal pair of~$\Gg$ and we define the equiva\-len\-ce
class~$[\Mm,\bs]$ to be the product of the equivalence classes~$[\Mm_j,\bs_j]$
(see Defini\-tion~\ref{DefELS}). 

The formal sum $\TT$, the group~$\Gg$ and the conjugacy class of
$\Mm\subseteq\Gg$ are uniquely determined by $(\BJ,\bl)$
(independently of the decomposition~$\bl=\bk\otimes\bs$). 
In fact, the group~$\Gg$ depends only on~$\TT$,
since~$\Gg_j\simeq\GL_{n'_j}(\kk_{\D'_j})$, where:
\begin{equation*}
n'_j\cdot[\kk_{\D'_j}:\kk_{\E_j}]=\frac{n_jd}{[\E_j:\F]} 
=\sum\limits_{i\in\I_j}\frac{m_id}{[\F[\b_i]:\F]},
\end{equation*}
which is the coefficient of $\TT_i$ in $\TT$, for $i\in\I_j$.

As in the previous case, the representation $\bk$ is not uniquely
determined by $\bl$, but once it is fixed (or, 
equivalently, once~$\k_{{\rm max},\mu}$ is fixed), it determines the
equivalence class $[\Mm,\bs]$. Further, there is a decomposed pair
$(\bjmax,\bkmax)$ above $(\J_{{\rm max},\mu},\k_{{\rm max},\mu})$ 
(see~\cite{MS1}) and we let~$\bjmax^1$ denote the 
pro-$p$ radical of~$\bjmax$; 
we are now in the situation of~\S\ref{SST}, with~$\textbf{\textsf{J}}=\bjmax$
and~$\bk=\bkmax$ so we have the functor:
\[
\KM=\KM_{\bkmax}:\Rr(\G)\to \Rr(\Gg),
\] 
which is also determined by the choice of~$\bk$. 
As in the homogeneous
case, every equivalence class~$[\Mm',\bs']$ arises from some
semisimple supertype~$(\BJ',\bl')$, by taking a cover.

We will see below
that~$\KM$ induces a bijection between the set of equivalence
classes~$[\BJ,\bl]$ of semi\-simple supertypes for~$\G$
such that~$\TT(\BJ,\bl)=\TT$ and the set of equivalence classes~$[\Mm,\bs]$
of super\-cuspidal pairs in~$\Gg$ (see Proposition~\ref{prop:three}); it might be
possible to prove this directly but in fact we deduce it as a
consequence of our block decomposition of~$\Rr(\G)$. 

\subsection{}

We continue with a semisimple supertype~$(\BJ,\bl)$ and all the
notation of the previous section, making a choice of
decomposition~$\bl=\bk\otimes\bs$. 
In particular we have Levi
subgroups~$\M\subseteq\L\subseteq\G$; a
decomposed pair~$(\bjmax,\bkmax)$ in~$\G$ 
of~$(\J_{{\rm max},\mu},\k_{{\rm max},\mu})$ in~$\L$;
a pair~$(\J_\a,\k_\a)$ in~$\M$; and a Levi subgroup~$\Mm$
of~$\Gg$. This gives us functors:
\begin{equation*}
\begin{split}
\KM=\KM_{\bkmax}&:\Rr(\G)\to\Rr(\Gg),\\
\KM_\L=\KM_{\k_{{\rm max},\mu}}&:\Rr(\L)\to\Rr(\Gg),\\
\KM_\M=\KM_{\k_\a}&:\Rr(\M)\to\Rr(\Mm),
\end{split}
\end{equation*}
using the notation of~\S\ref{SST}. 
Denote by~$\Q=\L\U$ the standard 
parabolic subgroup of~$\G$ with Levi component~$\L$, and by~$\Pp$ 
the standard parabolic subgroup of~$\Gg$ with Levi component~$\Mm$.

\begin{theo}
\label{Cagliostro}
For any smooth representation $\pi$ of $\M$, one has:
\begin{equation*}
\KM(\ia_\P^\G(\pi))\simeq\ia_{\Pp}^{\Gg}(\KM_\M(\pi)).
\end{equation*}
\end{theo}

\begin{proof}
First note that 
it is enough to prove the result when $\M=\L$. 
Indeed, assuming that the theorem is true for $\M=\L$, 
we set $\pi_0=\ia^{\L}_{\P\cap\L}(\pi)$ and get: 
\begin{equation*}
\KM(\ia_\P^\G(\pi))
\simeq\KM(\ia^\G_\Q(\pi_0))
\simeq\KM_\L(\ia^{\L}_{\P\cap\L}(\pi))
\end{equation*}
and the latter representation of $\Gg$ is isomorphic to 
$\ia_{\Pp}^{\Gg}(\KM_\M(\pi))$ thanks to Proposition \ref{MainTheoK}. 

Assume now that $\M=\L$.
Given $\pi\in\Rr(\L)$, by Lemma~\ref{L0}, we have an isomorphism:
\begin{equation*}
\Ind_\Q^{\Q\bjmax^{}}(\pi)\simeq\Ind_{\bjmax^{}\cap\Q}^{\bjmax^{}}(\pi)
\end{equation*}
of representations of $\bjmax^{}$. 
Since $\bjmax^{}=\bjmax^1(\bjmax^{}\cap\Q)$, we get: 
\begin{equation}
\label{F2M}
\KM(\Ind_{\Q}^{\Q\bjmax^{}}(\pi))
\simeq\Hom_{\bjmax^1\cap\Q}(\bk|_{\bjmax^{}\cap\Q},\pi)
\simeq\Hom_{\J^1_{{\rm max},\mu}}(\k_{{\rm max},\mu},\pi)
\end{equation}
which is $\KM_\L(\pi)$. 
Therefore it is enough to prove that:
\begin{equation}
\label{F3M}
\KM(\ia_\Q^\G(\pi))=\KM(\Ind_{\Q}^{\Q\bjmax^{}}(\pi))
\end{equation}
for all smooth representations $\pi$ of $\L$. 

First assume $\R$ is the field of complex numbers and~$\pi$ is irreducible.
Define a representation $\V$ of $\Gg$ by the following 
exact sequence:
\begin{equation}
\label{F1M}
0\to\KM(\Ind_{\Q}^{\Q\bjmax^{}}(\pi))\ffr{\iota}\KM(\ia^{\G}_\Q(\pi))\to\V\to0
\end{equation}
of representations of $\Gg$,
where $\iota$ is the inclusion map, and assume that $\V$ is nonzero. 
Then it has an irreducible subquotient, with some supercuspidal support 
$(\Mm',\bs')$. 
Let $\Pp'$ be the standard parabolic subgroup of $\Gg$ with Levi component 
$\Mm'$ and write $\Nn'$ for its unipotent radical. 
There is a standard parabolic subgroup $\P'=\M'\N'$ of 
$\G$ contained in $\Q$, having the following property:
the intersection $\P'\cap\L=\M'(\N'\cap\L)$ is a parabolic subgroup of~$\L$ such
that:
\begin{equation*}
(\U(\Lamax)\cap\mult\B\cap\N'\cap\L)(\U^1(\Lamax)\cap\mult\B)
/(\U^1(\Lamax)\cap\mult\B)=\Nn'.
\end{equation*}  
Let $[\La',n',0,\b]$ be a simple stratum such that: 
\begin{enumerate}
\item 
the image of~$\U^1(\La')\cap\mult\B\cap\L$ in~$\Gg$ is~$\Nn'$;
\item
$\U(\La')\cap\L\subseteq\U(\Lamax)\cap\L$ and 
$\U(\La')\cap\N'\cap\L=\U(\Lamax)\cap\N'\cap\L$
as in Lemma~\ref{L1}. 
\end{enumerate}
(Note that this makes sense because it is happening in~$\L$, 
where we just have a direct sum
of simple  strata so we can  do it separately  in each block of~$\L$  and then
take the sum.)

By using~\eqref{COHERENCE} and~\eqref{eqn:kP} in $\L$, 
there is an irreducible representation~$\k_{\P'\cap\L}$ 
of a group~$\J_{\P'\cap\L}$ which is compatible with~$\k_{\max,\mu}$, that is, 
we have an isomorphism: 
\[
\Ind_{\J_{\P'\cap\L}}^{(\U(\La')\cap\mult\B\cap\L)(\U^1(\La')\cap\L)} (\k_{\P'\cap\L}) 
\simeq
\Ind_{(\U(\La')\cap\mult\B\cap\L)\J^1_{\max,\mu}}^{(\U(\La')\cap\mult\B\cap\L) (\U^1(\La')\cap\L)} (\k_{\max,\mu}),
\]
and these induced representations are irreducible.  
In particular, by the Mackey formula, there is an 
element~$g\in (\U(\La')\cap\mult\B\cap\L)(\U^1(\La')\cap\L)$
that intertwines $\k_{\P'\cap\L}^{}$ with $\k_{\max,\mu}$. 
% \tb{(coming from a unique double-class, presumably the trivial one?)}, 
% \[
% \Hom_{\J_{\P'\cap\L}\cap((\U(\La')\cap\mult\B\cap\L)\J^1_{\max,\mu})^g}
% (\k_{\P'\cap\L}^{},\k_{\max,\mu}^g)\neq0.  
% \]
Moreover, the representation~$\k_{\P'\cap\L}$ 
is decomposed above the restriction of $\k_{\max,\mu}$ 
to $\J_{\P'\cap\L}\cap\L$, denoted $\k_{\L}$, which is a 
maximal~$\b$-extension of~$\J_{\L}$ in~$\L$.

By~\cite[Proposition~2.33]{MS1}, we get a representation~$\bk'$ of a compact open 
subgroup~$\BJ'$ which is decomposed above~$\k_{\P'\cap\L}$ in $\G$, so also 
above~$(\J_{\L},\k_{\L})$. 

\begin{lemm}[{cf.~\cite[Proposition~6.3]{BK1}}]
For~$i=1,2$, let~$\K_i$ be a subgroup of~$\G$ with an Iwahori decomposition 
with res\-pect to~$(\L,\Q)$, and let~$\rho_i$ be an irreducible representation 
of~$\K_i$ which is trivial on~$\U$ and~$\U^-$. 
Then, for~$g\in\L$, we have: 
\[
\Hom_{\K_1\cap(\K_2)^g}(\rho_1,(\rho_2)^g)=\Hom_{(\K_1\cap\L)\cap(\K_2\cap\L)^g}(\rho_1,(\rho_2)^g).
\]
\end{lemm}

\begin{proof}
One inclusion is obvious and the other follows from the fact 
that~$\K_1\cap(\K_2)^g$ has an Iwahori decomposition with respect to~$(\L,\Q)$.  
\end{proof}

Applying this lemma with~$\bk'$ and the restriction of~$\bkmax$ to 
${(\U(\La')\cap\mult\B)\BJ^1_{\max,\mu}}$, we see that $g$ intertwines 
these two representations. 
% $\bk'$ with $\bkmax|{(\U(\La')\cap\mult\B)\BJ^1_{\max,\mu}}$. 
Thus, by Mackey, there is a non-zero morphism:
\[
\Ind_{\BJ'}^{(\U(\La')\cap\mult\B\cap\L)(\U^1(\La')\cap\L)}(\bk') \to 
\Ind_{(\U(\La')\cap\mult\B)\bjmax^1}^{(\U(\La')\cap\mult\B\cap\L)(\U^1(\La')\cap\L)}(\bkmax).
\]
Moreover, the intertwining formula given by \cite[Proposition 2.31]{MS1}
(together with an analogue of \cite[Lemme 2.2]{MS1})
implies that both of these representations are 
irreducible.  
Thus they are isomorphic, and we have a compatibility property analogous 
to~\eqref{COHERENCE}. 

We now go back to \eqref{F1M}. 
By taking the $\Nn'$-fixed vectors and then the $\bs'$-isotypic component, 
and thanks to \eqref{F2M}, we get an exact sequence
\begin{equation*}
0\to\Hom_{\J^1_{{\rm max},\mu}}(\k_{{\rm max},\mu},\pi)^{\Nn',\bs'}\to
\Hom_{\bjmax^{1}}(\bkmax,\ia^{\G}_\P(\pi))^{\Nn',\bs'}\to\V^{\Nn',\bs'}\to0
\end{equation*}
% \begin{equation*}
% 0\to\Hom_{\J_{{\rm max},\mu}}(\k_{{\rm max},\mu}\otimes\rho,\pi)\to
% \Hom_{\bjmax^{}}(\bkmax\otimes\rho,\ia^{\G}_\P(\pi))\to(\V^{\Uu})^\rho\to0
% \end{equation*}
of complex vector spaces, which are finite-dimensional since~$\pi$ is 
admissible.
Now
\begin{eqnarray*}
\Hom_{\bjmax^1}(\bkmax,\Ind_\P^\G(\pi))^{\Nn',\bs'}
&\simeq&
\Hom_{(\U^1(\La')\cap\mult\B)\bjmax^1}(\bkmax,\Ind_\P^\G(\pi))^{\bs'} \\
&\simeq&
\Hom_{(\U(\La')\cap\mult\B)\bjmax^1}(\bkmax\otimes\bs',\Ind_\P^\G(\pi)) \\
&\simeq&
\Hom_{\BJ'}(\bk'\otimes\bs',\Ind_\P^\G(\pi)),
\end{eqnarray*}
where~$\bk'$ is compatible with~$\bkmax$ as above.
Similarly, we have
\[
\Hom_{\J^1_{{\rm max},\mu}}(\k_{{\rm max},\mu},\pi)^{\Nn',\bs'}
\simeq
\Hom_{\J_{\P'\cap\L}}(\k_{\P'\cap\L}\otimes\bs',\pi)
\]
Now, by \cite{SeSt2}, the semisimple type $\bl'=\bk'\otimes\bs'$ is a 
cover of $\k_{\P'\cap\L}\otimes\bs'$, which is itself a cover 
of~$\k_{\L}\otimes\bs'$.
Thus the algebra $\Hh=\End_\G(\ind_{\BJ'}^\G(\bk'\otimes\bs')$ is a free 
module of rank~$1$ over:
\begin{equation*}
\Hh_\L=\End_\L(\ind_{\J_{\P'\cap\L}}^\L(\k_{\P'\cap\L}\otimes\bs'))
\end{equation*}
(see \cite[Corollaire 2.32]{MS1}) and there is an isomorphism of~$\Hh$-modules
\[
\Hom_{\BJ'}(\bk'\otimes\bs',\ia^{\G}_\P(\pi))\simeq
\Hom_{\Hh_\L}(\Hh,\Hom_{\J_{\P'\cap\L}}(\k_{\P'\cap\L}\otimes\bs',\pi)).
\]
Since these are finite-dimensional, we deduce that~$\V^{\Nn',\bs'}=0$, 
a contradiction.

We deduce from Proposition~\ref{Annonce} that, for $g\in\G$, we have:
\begin{equation}
\label{Leila2M}
\Hom_{\bjmax^1\cap\U^g}(\bkmax,1)\ne 0
\quad\Leftrightarrow\quad
g\in\P\bjmax^{}.
\end{equation}
As $\bjmax^1$ is a pro-$p$-group, \eqref{Leila2M} also holds when~$\R$
has positive characteristic.
Thus, by Proposition~\ref{Annonce} again, the equality \eqref{F3M}
holds for all smooth $\R$-representations $\pi$ of $\L$.  
\end{proof}

%%%%%%%%%%%%%%%%%%%%%%%%%%%%%%%%%%%%%%%%%%%%%%%%%%%%%%%%%%%%%%%%%%%%%%%%
\section{A semisimple computation}
%\label{SSC}
%%%%%%%%%%%%%%%%%%%%%%%%%%%%%%%%%%%%%%%%%%%%%%%%%%%%%%%%%%%%%%%%%%%%%%%%

As in Section \ref{JJM}, the notation of which we use, 
$(\BJ,\bl)$ is a semisimple supertype of $\G$. 
We fix a decomposition $\bl=\bk\otimes\bs$ and write 
$\KM=\KM_{\bkmax}$ and $[\Mm,\bs]$ for the functor and the equivalence class 
of supercuspidal pairs associated with it.

\begin{prop}
\label{SemisimpleC}
Every irreducible subquotient of $\KM(\ind^{\G}_{\BJ}(\bl))$
has its supercuspidal support in $[\Mm,\bs]$. 
\end{prop}

\begin{proof}
Assume first that $(\BJ,\bl)$ is a maximal simple type.
Then $r=\ll=1$ and we have: 
\begin{equation*}
\KM(\ind^{\G}_{\BJ}(\bl))\simeq
\bigoplus\limits_{\BJ\backslash\G/\BJ}\ \KM(\ind^{\BJ}_{\BJ\cap\BJ^g}(\bl^g)).
\end{equation*}
By reciprocity, one see that the $g\in\G$ that contribute to this sum 
intertwine $\bn$. 
Therefore one may assume that they are in $\mult\B$. 
Since $\BJ\cap\mult\B$ is a maximal compact open subgroup in 
$\mult\B$, by the Cartan decomposition one may assume that the $g$ that contribute are diagonal 
matrices in $\mult\B$. 
As $\bs$ is cuspidal, only those $g$ which normalize $\BJ\cap\mult\B$
contribute to this sum. 
Fix $\w\in\mult\B$ such that the $\mult\B$-normalizer of 
$\BJ\cap\mult\B$ is generated by $\BJ\cap\mult\B$ and $\w$. 
We get: 
\begin{equation*}
\KM(\ind^{\G}_{\BJ}(\bl))\simeq
\bigoplus\limits_{n\in\ZZ}\ \KM(\bl^{\w^n})=
\bigoplus\limits_{\ZZ}\ (\bs\oplus\bs^{\phi}\oplus\dots\oplus\bs^{\phi^{b-1}})
=\bigoplus\limits_{\ZZ}\bigoplus\limits_{j=0}^{b-1}\ \bs^{\phi^{j}},
\end{equation*}
where $\phi$ is a generator of $\Gal(\kk_{\D'}/\kk_\E)$ and $b$ is the 
cardinality of the $\Gal(\kk_{\D'}/\kk_\E)$-orbit of $\bs$
(see \cite[Lemme 5.3]{MS1})

We treat the general case. 
Recall that we have the standard parabolic subgroup~$\P$ of~$\G$, with 
standard Levi component~$\M$. 
%\begin{equation*}
%\L=\L_1\tdt\L_s.
%\end{equation*}
We have an isomorphism:
\begin{equation*}
\ind^{\G}_{\BJ}(\bl)\simeq\ia^\G_\P(\ind^\M_{\BJ\cap\M}(\l_\a)).
\end{equation*}
As $\KM$ commutes with parabolic induction (see Theorem~\ref{Cagliostro}), 
we get: 
\begin{eqnarray*}
\KM(\ind^{\G}_{\BJ}(\bl))&\simeq&
\Ind^\Gg_{\Pp}\big(\KM_\M(\ind^\M_{\BJ\cap\M}(\l_\a))\big)\\
&\simeq&\Ind^\Gg_{\Pp}\big(\KM_1(\ind^{\G_{m_1}}_{\J_{1}}(\l_1))
\otimes\dots\otimes\KM_r(\ind^{\G_{m_r}}_{\J_{r}}(\l_r))\big)
\end{eqnarray*}
where we have $\KM_i=\KM_{\k_i}$.
For each $i\in\{1,\dots,r\}$ we have: 
\begin{equation*}
%\label{indcuspKinv}
\KM_i(\ind^{\G_i}_{\J_i}(\l_i))\simeq 
\bigoplus\limits_{\ZZ}\bigoplus\limits_{j=0}^{b_i-1}\ \s_i^{\phi_i^{j}},
\end{equation*}
where $\phi_i$ is a generator of $\Ga_i=\Gal(\kk_{\D'_i}/\kk_{\F[\b_i]})$ 
and $b_i$ is the cardinality of the orbit of $\s_i$ under $\Ga_i$. 
Thus: 
\begin{equation*}
\KM(\ind^{\G}_{\BJ}(\bl))\simeq
\bigoplus\limits_{\ZZ^r}\bigoplus\limits_{\boldsymbol{j}}\ 
(\s_1^{\phi_1^{j_1}}\times\dots\times\s_r^{\phi_r^{j_r}})
\end{equation*}
where $\boldsymbol{j}$ ranges over the $r$-tuples $(j_1,\dots,j_r)$ with 
$j_i\in\{0,\dots,b_i-1\}$ for all $i\in\{1,\dots,r\}$, and where $\times$ 
stands for parabolic induction. 
The result follows by unicity of supercuspidal support in~$\Gg$.
\end{proof}

%%%%%%%%%%%%%%%%%%%%%%%%%%%%%%%%%%%%%%%%%%%%%%%%%%%%%%%%%%%%%%%%%%%%%%%%
\section{Supercuspidal inertial classes and supertypes}
\label{bijomjl}
%%%%%%%%%%%%%%%%%%%%%%%%%%%%%%%%%%%%%%%%%%%%%%%%%%%%%%%%%%%%%%%%%%%%%%%%

Given $(\BJ,\bl)$ a semisimple supertype of $\G$, 
write $\Irr(\BJ,\bl)$ for the set of all classes of irreducible 
sub\-quotients of $\ind^\G_\BJ(\bl)$. 

Given $\Om$ an inertial class of supercuspidal pairs of $\G$, write
$\Irr(\Om)$
for the set of all classes of irreducible 
representations of $\G$ having their supercuspidal support in $\Om$. 

\begin{prop}
\label{AdamTrask}
Let $(\M,\vr)$ be a supercuspidal pair of $\G$
and $(\BJ,\bl)$ be a semisimple supertype of $\G$ 
associated with a maximal simple type $(\J_\a,\l_\a)$ 
of $\M$ contained in $\vr$. 
Write $\Om$ for the inertial class of $(\M,\vr)$. 
Then we have $\Irr(\Om)=\Irr(\BJ,\bl)$.
\end{prop}

\begin{proof}
We begin by proving the containment~$\Irr(\Om)\subseteq\Irr(\BJ,\bl)$.
Assume $\M$ is standard and write $\vr=\rho_1\otimes\dots\otimes\rho_r$, 
where $\rho_i$ is a supercuspi\-dal irreducible representation of $\G_{m_i}$ for 
$m_i\>1$. 
For $i\in\{1,\dots,r\}$, fix an unramified character $\chi_i$ of $\G_{m_i}$. 
Then $\rho_i\chi_i$ is a quotient of the compact induction of $\l_i$ to $\G_{m_i}$. 
% $\ind_{\J_i}^{\G_{m_i}}(\l_i)$. 
It follows that $\rho_1\chi_1\tdt\rho_r\chi_r$ is a quotient of: 
\begin{equation}
\label{MrDarcy}
\ind^{\G_{m_1}}_{\J_1}(\l_1)\tdt\ind^{\G_{m_r}}_{\J_1}(\l_r) 
\simeq\ind^\G_\BJ(\bl).
\end{equation}
Thus any irreducible sub\-quo\-tient of $\rho_1\chi_1\tdt\rho_r\chi_r$ appears 
in $\Irr(\BJ,\bl)$.

For the opposite containment, we need the following lemma.

\begin{lemm}
\label{Cont}
Let $\Om$ and $(\BJ,\bl)$ be as in Proposition~\ref{AdamTrask},
and assume that $\Irr(\BJ,\bl)$ contains a cuspidal representation $\pi$. 
Then we have $\pi\in\Irr(\Om)$.
\end{lemm}

\begin{proof}
Let $(\J_0,\l_0)$ be a maximal simple type of $\G$ contained in $\pi$. 
It is attached to a simple stratum $[\La_0,n_0,0,\b_0]$ and we write $\t_0$ 
for the simple character occurring in the restriction of $\l_0$ to
$\H_0^1=\H^1(\b_0,\La_0)$. 
This character occurs as a subquotient (hence a subrepresentation since $\H_0^1$ 
is a pro-$p$ group) 
of the restriction of $\ind^\G_\BJ(\bl)$ to $\H_0^1$.
Recall that we have an isomorphism \eqref{MrDarcy} and that 
the compact induction of $\l_i$ to $\G_{m_i}$ is isomorphic to
\begin{equation*}
\rho_i\otimes\R[\X,\X^{-1}],
\end{equation*}
with $\G_{m_i}$ acting on $\R[\X,\X^{-1}]$ by $g\cdot\X^{k}=\X^{k+v(g)}$, for 
all $k\in\ZZ$, where $v(g)$ is the valuation of the reduced norm of 
$g\in\G_{m_i}$. 
Therefore, when restricting~\eqref{MrDarcy} to $\H_0^1$, 
we deduce that $\t_0$ occurs as a subrepresentation of
\begin{equation*}
\bigoplus\limits_{\ZZ^r}\ (\rho_1\tdt\rho_r).
\end{equation*}
Thus $\t_0$ occurs as a subrepresentation of $\rho_1\tdt\rho_r$, and it 
follows from \cite[Proposition 5.6]{MS1} that the sum $\TT=\TT(\BJ,\bl)$ is equal to
\begin{equation*}
\TT(\J_0,\l_0)=\frac{md}{[\F[\b_0]:\F]}\cdot\TT_0, 
\end{equation*}
where $\TT_0$ is the endo-class of $\pi$.
We thus are in the homogeneous situation of Section~\ref{ss:hom}
so that a decomposition $\bl=\bk\otimes\bs$ is determined by
a pair $(\jmax,\kmax)$. 
Then the simple character~$\tmax$ contained in~$\kmax$ is the transfer of
the simple character~$\t_0$ in~$\l_0$.

We fix a decomposition $\l_0=\k_0\otimes\s_0$ and write $\KM_0=\KM_{\k_0}$. 
By \cite{BSS}, 
the characters~$\t_0$ and~$\tmax$ are in fact conjugate and, replacing
the pair~$(\BJ,\bl)$ by a suitable~$\G$-conjugate, we may assume that
the pairs $(\jmax,\kmax)$ and $(\J_0,\k_0)$ coincide.
Thus the functor $\KM=\KM_{\kmax}$ of section~\ref{ss:hom} coincides with~$\KM_0$.
This also induces a decomposition $\l_i=\k_i\otimes\s_i$ for all 
$i\in\{1,\dots,r\}$.

We now apply this functor to the subquotient $\pi$ of 
$\ind^\G_{\BJ}(\bl)$.
By \cite[Lemme 5.3]{MS1}, the representation $\KM(\pi)$ is a sum of cuspidal 
irreducible representations of $\Gg=\GL_{m'}(\kk_{\D'})$. 
By Proposition \ref{SemisimpleC}, these cuspidal representations have 
their supercuspidal support in $[\Mm,\bs]$. 
By the classification of cuspidal irreducible representations of $\Gg$ in 
terms of supercuspidal representations 
(see for instance 
% \cite{James} or 
\cite[Proposition 3.7]{MS2}), 
there is a supercuspidal representation $\s$ of 
$\GL_{m'/r}(\kk_{\D'})$ such that
\begin{equation*}
\s_i=\s^{\g_i},
\quad
\g_i\in\Gal(\kk_{\D'}/\kk_{\F[\b_0]}),
\quad
i\in\{1,\dots,r\},
\end{equation*}
and an integer $u\>0$ such that we have $r=e(\s)\ell^u$, 
where $e(\s)$ is a positive integer attached to $\s$ 
(see \cite[Remarque 3.6]{MS2}).
Since $\k_i\otimes\s$ can be obtained from $\l_i$ by conjugacy in 
$\G_{m_i}$, we may assume without changing $\ind^\G_{\BJ}(\bl)$ that 
we have:
\begin{equation*}
\TT_i=\dots=\TT_r=\TT_0,
\quad
\s_1=\dots=\s_r=\s.
\end{equation*}
By \cite[Corollaire 5.5]{MS1}, it follows that $\rho_1,\dots,\rho_r$ are 
inertially equivalent to a given supercuspidal representation $\rho$. 
It also follows from \cite[\S6]{MS2} 
% (see the first part of the proof)
that $\pi$ is inertially equivalent to 
$\St(\rho,r)$, the unique cuspidal irreducible subquotient of the product 
$\rho\times\rho\nu_\rho^{}\tdt\rho\nu_\rho^{r-1}$
(where $\nu_\rho$ is the unramified character associated with $\rho$ in 
\cite[\S4.5]{MS1}).
It follows that the supercuspidal pair 
$(\M,\vr)$ is inertially equivalent to 
$(\M,\rho\odo\rho)$ and that $\pi$ appears in $\Irr(\Om)$.
\end{proof}

We return to the proof of Proposition~\ref{AdamTrask}. 
Let $\pi$ be an irreducible subquotient of $\ind^\G_\BJ(\bl)$, and let 
$(\L,\tau)$ be its cusp\-idal support. 
Write:
\begin{equation*}
% \label{EQ4}
\ind^\G_\BJ(\bl)\simeq\ia^\G_\P(\ind^\L_{\J_\a}(\l_\a))
=\ind^{\G_{m_1}}_{\J_1}(\l_1)\tdt\ind^{\G_{m_r}}_{\J_r}(\l_r).
\end{equation*}
For $i\in\{1,\dots,r\}$, 
note that $\Pi_i=\ind^{\G_{m_i}}_{\J_i}(\l_i)$ is made of supercuspidal irreducible 
subquotients all of whose are unramified twists of a given supercuspidal 
irreducible representation $\rho_i$ of $\G_{m_i}$. 
Let $\Q=\L\U$ be a parabolic subgroup of $\G$ with Levi component $\L$. 
We compute the Jacquet module $(\ind^\G_\BJ(\bl))_{\U}$.
Since it contains $\pi_\U$, it contains an irreducible cuspidal 
subquotient which is $\G$-conjugate to $\tau$. 
By the geometric lemma, there are a permutation $w$ of $\{1,\dots,r\}$ and 
integers $0=a_0<a_1<\dots<a_t=r$ such that, 
if we write $\tau=\tau_1\odo\tau_t$ with $\tau_j$ 
cuspidal, then $\tau_j$ appears, for each $j\in\{1,\dots,t\}$, 
as a subquotient of:
\begin{equation*}
\Sigma_j=\Pi_{w(a_{j-1}+1)}\tdt\Pi_{w(a_j)}.
\end{equation*}
It follows from Lemma \ref{Cont} that $\tau_j$ has its supercuspidal support
in $\Om_j$, the inertial class of the supercuspidal pair: 
\begin{equation*}
(\G_{w(a_{j-1}+1)}\tdt\G_{w(a_j)},\rho_{w(a_{j-1}+1)}\odo\rho_{w(a_j)}).
\end{equation*}
It follows that $\pi$ has its supercuspidal support in $\Om$, as required. 
\end{proof}

\begin{prop}%\label{prop:equiv}
Let $(\BJ,\bl)$ and $(\BJ',\bl')$ be semisimple supertypes of~$\G$. 
The representations $\ind^{\G}_{\BJ'}(\bl')$, $\ind^\G_\BJ(\bl)$ have an 
irreducible subquotient in common if and only if $[\BJ,\bl]=[\BJ',\bl']$. 
\end{prop}

\begin{proof}
Since the~$\Irr(\Om)$ form a partition of the set of all isomorphism
classes of irreducible representations of~$\G$, it follows from
Proposition~\ref{AdamTrask} that $\ind^{\G}_{\BJ'}(\bl')$,
$\ind^\G_\BJ(\bl)$ have an irreducible subquotient in common if and
only if~$\Irr(\BJ,\bl)=\Irr(\BJ',\bl')$. 

Suppose that~$\Irr(\BJ,\bl)=\Irr(\BJ',\bl')=\Irr(\Om)$, with~$\Om=[\M,\vr]_\G$. 
If~$\M=\G$ then, by following the proof of Lemma~\ref{Cont}, we find 
that~$(\BJ,\bl)$ and~$(\BJ',\bl')$ are both equivalent to maximal 
simple supertypes; 
by unicity (up to conjugacy) of maximal simple supertypes in a 
supercuspidal representation (see~\cite[Th\'eor\`eme~3.11]{MS1} and 
\cite[Proposition~6.10]{MS2}), 
we deduce that~$[\BJ,\bl]=[\BJ',\bl']$. 
In the general case, we have
\[
\ind^\G_\BJ(\bl)\simeq\ia^\G_\Q(\ind^\M_{\J_\a}(\l_\a))
\simeq\ia^\G_\Q(\ind^\M_{\J'_\a}(\l'_\a))\simeq\ind^\G_{\BJ'}(\bl'),
\]
where the middle isomorphism follows from the previous case.
\end{proof}

It also follows that there is a bijection:
\begin{equation}
\label{BIJOMSST}
\Om\leftrightarrow[\BJ,\bl]
\end{equation}
between inertial classes of supercuspidal pairs of $\G$ and 
equivalence classes of semisimple super\-types of $\G$, 
characterized by the equality 
$\Irr(\Om)=\Irr(\BJ,\bl)$.
% \end{coro}

%%%%%%%%%%%%%%%%%%%%%%%%%%%%%%%%%%%%%%%%%%%%%%%%%%%%%%%%%%%%%%%%%%%%%%%%
\section{Splitting of the category}
\label{S7}
%%%%%%%%%%%%%%%%%%%%%%%%%%%%%%%%%%%%%%%%%%%%%%%%%%%%%%%%%%%%%%%%%%%%%%%%

Let $(\BJ,\bl)$ be a semisimple supertype of $\G$, together with 
a decomposition $\bl=\bk\otimes\bs$.
Asso\-ciated with it, there are a formal sum $\TT$ of endo-classes, 
a functor $\KM=\KM_{\bkmax}$ and the group~$\Gg=\bjmax^{}/\bjmax^1$. 

\subsection{}

We now fix $\TT$ and $\KM$, and make $[\Mm,\bs]$ vary among the
equivalence classes of supercuspidal pairs of $\Gg$. By
Corollary~\ref{cor:decomp}, we have, for all  $\V\in\Rr(\G)$, a
decomposition:
\begin{equation}
\label{GrFini}
\KM(\V)=\bigoplus\limits_{[\Mm,\bs]}\V(\TT,\bs),
\end{equation}
where~$\V(\TT,\bs)$ is the maximal subspace of $\KM(\V)$ all of whose 
composition factors have supercuspidal support in $[\Mm,\bs]$. 

\begin{defi}
Given $\V\in\Rr(\G)$ a smooth representation, we write:
\begin{enumerate}
\item $\V[\TT,\bs]$ for the $\G$-subspace of $\V$ generated by 
$\V(\TT,\bs)$;
\item $\V[\TT]$ for the $\G$-subspace of $\V$ generated by 
$\KM(\V)$.
\end{enumerate}
\end{defi}

Thus $\V[\TT]$ is the sum of all the $\V[\TT,\bs]$, as $[\Mm,\bs]$ ranges 
over the set of equivalence classes of supercuspidal pairs of $\Gg$.
We claim that $\V[\TT]$ is in fact the direct sum of the $\V[\TT,\bs]$. 

\begin{lemm}
\label{Dupuys}
Given $[\Mm,\bs]$, $[\Mm',\bs']$ equivalence classes of supercuspidal pairs of $\Gg$, we have:
\begin{equation*}
\V[\TT,\bs] (\TT,\bs')=
\left\{
\begin{array}{ll}
\V(\TT,\bs) & \text{ if } [\Mm',\bs']=[\Mm,\bs];\\
0 & \text{ otherwise.}
\end{array}
\right.
\end{equation*}
\end{lemm}

\begin{proof}
We have the containment 
$\V[\TT,\bs](\TT,\bs)\subseteq\V(\TT,\bs)$.  
Since $\V[\TT,\bs]$ contains $\V(\TT,\bs)$, this containment 
is an equality. 
Write $\TM$ for the functor 
$\xi\mapsto\KM(\ind^\G_{\bjmax^{}}(\bkmax\otimes\xi))$. 
We have a surjective map: 
\begin{equation*}
\ind^\G_{\bjmax^{}}(\bkmax\otimes\V(\TT,\bs))\to\V[\TT,\bs]
\end{equation*}
thus a surjective map:
\begin{equation*}
\TM(\V(\TT,\bs))\to\KM(\V[\TT,\bs]). 
\end{equation*}
To prove the remaining part of the lemma, 
it is enough to prove that any irreducible subquo\-tient of 
the left hand side has super\-cus\-pidal support in $[\Mm,\bs]$. 
As $\TM$ is exact, 
it is enough to prove that, for all irreducible representation $\xi$ 
with supercuspidal support in $[\Mm,\bs]$, 
any irreducible subquotients of $\TM(\xi)$ 
has supercuspidal support in $[\Mm,\bs]$. 
By the same exactness argument, it is enough to prove the following lemma. 

\begin{lemm}%\label{lem:Ls}
Let $(\Mm',\bs')\in[\Mm,\bs]$ and $\X=\Ind^\Gg_{\Mm'}(\bs')$.
Then all irreducible subquotients of $\TM(\X)$ 
have supercuspidal support in $[\Mm,\bs]$. 
\end{lemm}

\begin{proof}
We may and will assume that $\Mm'=\Mm$. 
We see $\bs'$ as a representation of $\BJ$ trivial on $\BJ^1$ and 
write $\bl'$ for the semisimple supertype $\bk\otimes\bs'$. 
Then we have:
\begin{equation*}
\ind^\G_{\bjmax^{}}(\bkmax\otimes\X)\simeq\ind^\G_{\BJ}(\bk\otimes\bs')=\ind^\G_{\BJ}(\bl').
\end{equation*}
Then the lemma follows from Proposition \ref{SemisimpleC}. 
\end{proof}
This ends the proof of Lemma \ref{Dupuys}. 
\end{proof}

As a corollary, we have the following result.

\begin{coro}
\label{COR}
For all smooth representations $\V$ of $\G$, we have: 
\begin{equation*}
\V[\TT]=\bigoplus\limits_{[\Mm,\bs]}\V[\TT,\bs]. 
\end{equation*}
\end{coro}

\begin{rema}
Note that, given $\V\in\Rr(\G)$, 
the subrepresentation $\V[\TT]$ does not depend on the 
choice of the functor $\KM$;
a different choice of~$\bk$ simply permutes the equivalence
classes of supercuspidal pairs~$[\Mm,\bs]$ so permutes the
terms~$\V[\TT,\bs]$ in $\V[\TT]$.
\end{rema}

\subsection{}
We now make $\TT$ vary among all possible formal 
sums of endo-classes arising from a semisimple supertype of $\G$. 

\begin{theo}
\label{TheoTTs}
For all smooth representation $\V$ of $\G$, there is an isomorphism:
\begin{equation*}
\V\simeq\bigoplus\limits_{\TT}\V[\TT]
\end{equation*}
of representations of $\G$. 
\end{theo}

\begin{proof}
Let $\V$ be a smooth representation of $\G$. 
We have a morphism:
\begin{equation*}
f:\bigoplus\limits_{\TT}\V[\TT]=\Y\to\V.
\end{equation*}
Write $\W$ for its kernel. 

\begin{lemm}
We have:
\begin{equation*}
\W=\bigoplus\limits_{\TT}(\W\cap\V[\TT]).
\end{equation*}
\end{lemm}

\begin{proof}
Let $\Z$ denote the quotient of $\W$ by the right hand side, and assume that 
it is nonzero. 
Let $\pi$ be an irreducible subquotient of $\Z$.
For all sums of endo-classes $\TT$, 
the representation $\pi$ is an ir\-reducible sub\-quo\-tient of 
$\W/(\W\cap\V[\TT])$, thus of:
\begin{equation*}
\V/\V[\TT]=
\bigoplus\limits_{\TT'\neq\TT}\V[\TT'],
\end{equation*}
which implies that $\pi[\TT]=0$. 
Since $\pi$ contains some semisimple supertype $(\BJ,\bl)$ by \cite{SeSt2,MS1}, 
for any decomposition~$\bl=\bk\otimes\bs$ with associated functor~$\KM$ and 
formal sum~$\TT$, we have~$\KM(\pi)\ne 0$ so that~$\pi[\TT]\ne 0$, a 
contradiction. 
\end{proof}

Since $f$ is injective on each $\V[\TT]$, 
we have $\W\cap\V[\TT]=0$ for all $\TT$
and it follows that we have $\W=0$. 
Assume that $f$ is not surjective, and let $\pi$ be an irreducible 
subquo\-tient in its cokernel.  
Write $\Om$ for the inertial class of its supercuspidal support.
Its corresponds to some semisimple supertype $(\BJ,\bl)$. 
Write $\TT=\TT(\BJ,\bl)$ and fix a decomposition 
$\bl=\bk\otimes\bs$.
By applying $\KM$, we get that $\KM(\pi)$ is 
a subquotient of:
\begin{equation*}
\KM(\V)/\KM(\Y)
=\KM(\V)/\KM(\V[\TT])
=\KM(\V)/\bigoplus\limits_{[\Mm,\bs]}\KM(\V[\TT,\bs])
=\KM(\V)/\bigoplus\limits_{[\Mm,\bs]}\V(\TT,\bs)
\end{equation*}
by Corollary \ref{COR} and Lemma \ref{Dupuys}.
But the right hand side is zero by \eqref{GrFini}: contradiction. 
\end{proof} 

%%%%%%%%%%%%%%%%%%%%%%%%%%%%%%%%%%%%%%%%%%%%%%%%%%%%%%%%%%%%%%%%%%%%%%%%
\section{Blocks of the category}
\label{S8}
%%%%%%%%%%%%%%%%%%%%%%%%%%%%%%%%%%%%%%%%%%%%%%%%%%%%%%%%%%%%%%%%%%%%%%%%

\begin{defi}
A \textit{block} in $\Rr(\G)$ is a full abelian subcategory which 
cannot be de\-com\-posed into two full abelian subcategories, and which is a 
direct summand in $\Rr(\G)$. 
\end{defi}

\subsection{}
Given $\Om$ an inertial class of a supercuspidal pair of $\G$, 
we write $\Rr(\Om)$ for the full subcategory of representations all of 
whose irreducible subquotients have their supercuspidal support in $\Om$. 

Given $(\BJ,\bl)$ a semisimple supertype of $\G$, we fix a decomposition 
$\bl=\bk\otimes\bs$ and associate to it the sum $\TT$, the functor $\KM=\KM_{\bkmax}$ 
and the equivalence class $[\Mm,\bs]$.
We write $\Rr(\BJ,\bl)$ for the full subcategory of representations 
$\V\in\Rr(\Om)$ such that $\V=\V[\TT,\bs]$. 
This does not depend on the choice of the decomposition of $\bl$. 

Assume that $\Om=[\L,\vr]_\G$ and $[\BJ,\bl]$ correspond to each other
(see Section~\ref{bijomjl}).

\begin{prop}
\label{Mazda}
One has $\Rr(\Om)=\Rr(\BJ,\bl)$. 
\end{prop}

\begin{proof}
Given $\V\in\Rr(\Om)$, we apply Theorem \ref{TheoTTs} and thus get a 
decomposition: 
\begin{equation}
\label{EEustace}
\V=\bigoplus\limits_{\TT'}\V[\TT'].
\end{equation}
Assume $\V[\TT']$ is nonzero for some sum $\TT'$,
and let $\W$ be an irreducible sub\-quo\-tient of it. 
Note that $\W$ has supercuspidal support in $\Om$.
We first prove that $\TT'=\TT$.
For this, it is enough to prove the following lemma.

\begin{lemm}
We have $\KM(\W)\neq0$. 
\end{lemm}

\begin{proof}
If $\Om$ is homogeneous, that is, 
if $\Om$ is the inertial class of a tensor product of 
copies of a given supercuspidal representation, the result 
is given by \cite[Proposition 5.8]{MS1}.
In general, we use \cite[Théorème 8.19]{MS2} together with 
Theorem \ref{Cagliostro} to reduce to the homogeneous case. 
\end{proof}

We thus have $\TT'=\TT$, and $\KM(\W)$ is a subquotient of:
\begin{equation*}
\KM(\V[\TT])=\bigoplus\limits_{[\Mm',\bs']}\V(\TT,\bs'). 
\end{equation*}
But there is also an unramified character $\chi$ of $\L$ such that 
$\KM(\W)$ is a subquotient of:
\begin{equation*}
\KM(\ia^\G_\Q(\vr\chi))\simeq\Ind^\Gg_{\Mm}(\KM_\L(\vr\chi)),
\end{equation*}
which is a finite direct sum of representations of the form 
$\Ind^\Gg_{\Mm'}(\bs')$ for $(\Mm',\bs')\in[\Mm,\bs]$. 
Thus all irreducible subquotients of $\KM(\W)$ have supercuspidal support in 
$[\Mm',\bs']$, and the de\-com\-po\-sition 
\eqref{EEustace} reduces to $\V=\V[\TT,\bs]$. 
Conversely, let $\V\in\Rr(\BJ,\bl)$ and let $\W$ be an irreducible 
subquotient of $\V$. 
All irreducible subquotients of $\KM(\W)$ have supercuspidal support 
in $[\Mm,\bs]$. 
Write $\h$ for the canonical surjective map: 
\begin{equation*}
\ind^\G_{\jmax^{}}(\kmax\otimes\KM(\W))\to\W.
\end{equation*}
Choose a composition series 
$0=\Z_0\subsetneq\Z_1\subsetneq\dots\subsetneq\Z_n=\KM(\W)$ 
and write $\W_i=\ind^\G_{\jmax^{}}(\kmax\otimes\Z_i)$. 
There is a minimal $i$ such that $\h$ is nonzero on $\W_{i+1}$. 
Thus $\W$ is isomorphic to a quotient of:
\begin{equation*}
\W_{i+1}/\W_i\simeq\ind^\G_{\jmax^{}}(\kmax\otimes(\Z_{i+1}/\Z_i))
\end{equation*}
and $\Z_{i+1}/\Z_i$ has supercuspidal support in $[\Mm,\bs]$. 
Thus $\W$ is a sub\-quotient of $\ind_\BJ^\G(\bl)$.
Now the result follows from Proposition \ref{AdamTrask}. 
\end{proof}

\subsection{}
Theorem~\ref{TheoTTs} and Corollary~\ref{COR} can now be restated as follows.

\begin{theo}
\label{TheoTTsOm}
The category $\Rr(\G)$ decomposes into the product of the 
subcategories $\Rr(\Om)$, where $\Om$ ranges over 
all possible inertial classes of supercuspidal pairs of $\G$.
\end{theo}

The following result says that the decomposition given by 
Theorem \ref{TheoTTsOm} is the best possible. 

\begin{prop}
%\label{TheDriver}
Each subcategory $\Rr(\Om)$ is indecomposable.  
\end{prop}

\begin{proof}
Assume this is not the case. 
There are two subcategories $\Aa$ and $\Aa'$ such that:
\begin{equation*}
\Rr(\Om)=\Aa\oplus\Aa'.
\end{equation*}
Let $[\BJ,\bl]$ be the equivalence class of semisimple supertypes 
which cor\-res\-ponds to 
$\Om$ and consider $\V=\ind^\G_\BJ(\bl)$. 
By Proposition \ref{Mazda}, we have $\V\in\Rr(\Om)$,
and there is a decomposition $\V=\W\oplus\W'$ with $\W\in\Aa$ and 
$\W'\in\Aa'$, and with no nonzero intertwining between $\W$ and $\W'$.
We get: 
\begin{equation*}
\End_\G(\V)=\End_\G(\W)\oplus\End_\G(\W'). 
\end{equation*}
This implies that $\End_\G(\V)$ possesses a nontrivial central idempotent. 
By \cite{SeSt2,MS1}, 
this algebra is isomorphic to a finite tensor product of 
affine Hecke algebras $\Hh(n_i,q^{f_i})$, with $1\<i\<r$. 
Thus its centre is isomorphic to the finite tensor product of the centres of 
the algebras $\Hh(n_i,q^{f_i})$, with $1\<i\<r$. 
The centre of $\Hh(n,q^{f})$ is isomorphic to 
$\CR[\X_1^{\pm1},\dots,\X_{n}^{\pm1}]^{\mathfrak{S}_{n}}$,
where $\mathfrak{S}_{n}$ is the $n$th symmetric group acting on 
$\X_1,\dots,\X_{n}$. 
This is an integral domain.
Thus the centre of $\End_\G(\V)$ does not contain any 
nontrivial idempotent.  
Therefore $\W'$, say, is zero. 
Now let $\X$ be a simple object in $\Aa'$. 
There is a $\G$-subspace $\Y$ of $\V$ such that $\X$ is a quotient of $\Y$. 
As $\V\in\Aa$, we get $\Y\in\Aa$. 
But $\Hom(\Y,\X)$ is nonzero: contradiction. 
\end{proof}

\subsection{}

Let~$\pi$ be a supercuspidal irreducible representation of~$\G$.
The endo-class of a simple character in~$\pi$ is well-defined 
(see~\cite[\S9.2]{BSS}) and we denote it~$\TT_{\pi}$.
Moreover, if $(\BJ,\bl)$ is a maximal simple supertype of~$\G$ occurring in $\pi$ 
and attached to a simple stratum $[\La,n,0,\b]$, then we have: 
\begin{equation*}
\TT(\BJ,\bl)=\frac{md}{[\F[\b]:\F]}\cdot\TT_\pi.
\end{equation*}
It does not depend on the choice of the simple type $(\BJ,\bl)$ nor of 
the simple stratum $[\La,n,0,\b]$, and we denote it $\TT(\pi)$. 
In fact, it depends only on the inertial class~$[\G,\pi]_\G$.

Now let~$\Om$ be the inertial class of a supercuspidal pair~$(\M,\vr)$ of 
$\G$.
We may (and will) assume that~$\M=\G_{m_1}\times\cdots\times\G_{m_r}$ and 
$\vr=\rho_1\otimes\cdots\otimes\rho_r$ with $m_1+\dots+m_r=m$ and 
$\rho_i$ an irreducible supercuspidal representation of $\G_{m_i}$, for each 
$i\in\{1,\dots,r\}$.
Then the formal sum:
\[
\TT(\Om)=\sum_{i=1}^r
% m'_i
\TT(\rho_i)
\]
is well-defined. 
Moreover, if $(\BJ,\bl)$ is a semisimple supertype of $\G$ such that 
$[\BJ,\bl]$ corresponds to $\Om$ in the sense of \eqref{BIJOMSST}, then we 
have $\TT(\BJ,\bl)=\TT(\Om)$.

\begin{prop}\label{prop:three}
Let~$(\BJ_0,\bl_0)$ be a semisimple supertype, put~$\TT=\TT(\BJ_0,\bl_0)$ 
and write $\Gg$ for the finite reductive group associated with it. 
Then the following finite sets have the same cardinality: 
\begin{enumerate}
\item\label{pt:i}
the set of supercuspidal inertial classes~$\Om$ of~$\G$ 
with~$\TT(\Om)=\TT$; 
\item\label{pt:ii}
the set of equivalence classes~$[\BJ,\bl]$ of semisimple 
supertypes of~$\G$ with~$\TT(\BJ,\bl)=\TT$; 
\item\label{pt:iii}
the set of equivalence classes~$[\Mm,\bs]$ of supercuspidal pairs in~$\Gg$.
\end{enumerate}
Moreover any choice of functor~$\KM$ associated with~$(\BJ_0,\bl_0)$ 
induces a bijection between the sets in~\eqref{pt:ii} 
and~\eqref{pt:iii}. 
\end{prop}

\begin{proof}
We have already seen the bijection between the first two sets. 
We make a choice of a functor $\KM$ 
associated with $(\BJ_0,\bl_0)$. 
We have already seen that~$\KM$ induces a surjective map from 
the set in~\eqref{pt:ii} to that in~\eqref{pt:iii}.
Thus it is enough to check that the sets in~\eqref{pt:i} and~\eqref{pt:iii} have the same
cardinality. 
Moreover, it is enough to treat the case where~$\TT$ is homogeneous, thus
\begin{equation*}
\TT=\frac{md}{[\E:\F]}\cdot\TT_1=m'd'\cdot\TT_1
\end{equation*}
as in~\S\ref{ss:hom}.

By the unicity (up to conjugacy) of maximal
simple supertypes in a supercuspidal representa\-tion
(see~\cite[Theorem 7.2]{SeSt2} and also~\cite[Corollaire 5.5]{MS1}),
the number of inertial classes~$[\G,\pi]_\G$ of supercuspidal 
representations with a given endo-class~$\TT_1$ 
is precisely the number of~$\Gal(\kk_{\D'}/\kk_\E)$-conjugacy classes
of supercuspidal re\-presen\-tations of~$\GL_{m'}(\kk_{\D'})$, where the
notation is as in~\S\ref{badm}.

We think of an inertial class of supercuspidal pairs of $\G$ as a finitely
supported map:
\begin{equation*}
\phi:\bigcup\limits_{k\>1}\{\text{inertial classes~$[\G_k,\pi]_{\G_k}$ of supercuspidal irreducible 
representations of $\G_k$}\}\to\NN
\end{equation*}
such that
\[
\sum_{k\ge 1} k \sum_{[\G_k,\pi]} \phi([\G_k,\pi]_{\G_k}) = m.
\]
We deduce that the number of inertial classes of supercuspidal pairs~$\Om$ 
with a given homogeneous $\TT$ is precisely the number of finitely supported maps:
\begin{equation*}
\psi:\bigcup\limits_{f\>1}\{\text{$\Gal(\kk_{\D'}/\kk_\E)$-conjugacy classes~$[\s]$
of supercuspidal re\-pre\-sentations of~$\GL_{f}(\kk_{\D'})$}\}\to\NN
\end{equation*}
such that
\[
\sum_{f\ge 1} f \sum_{[\s]} \psi([\s]) = m',
\]
where we are again using the notation of~\S\ref{badm}.
But this is also the number of equivalence classes of supercuspidal pairs 
in~$\Gg=\GL_{m'}(\kk_{\D'})$. 
\end{proof}

%%%%%%%%%%%%%%%%%%%%%%%%%%%%%%%%%%%%%%%%%%%%%%%%%%%%%%%%%%%%%%%%%%%%%%%%
\section{A remarkable property of supercuspidal representations}
\label{S9}
%%%%%%%%%%%%%%%%%%%%%%%%%%%%%%%%%%%%%%%%%%%%%%%%%%%%%%%%%%%%%%%%%%%%%%%%

We end this article by the following result.  
When $\G$ is split, that is when $\G=\GL_n(\F)$, $n\>1$, 
it is proven by Dat \cite[Corollaire B.1.3]{DatLTNAE} in a different manner. 

\begin{prop}
\label{DDSC}
Let $\P$ be a proper parabolic subgroup of $\G$ 
and $\s$ be a representation of a Levi component $\M$ of $\P$.
Then $\Ind^\G_\P(\s)$ has no supercus\-pi\-dal irreducible 
subquotient.
\end{prop}

\begin{proof}
When $\s$ is irreducible, the result follows from the definition of a 
supercus\-pi\-dal representation (Definition~\ref{DefSCP}).
Assume $\Ind^\G_\P(\s)$ contains a supercus\-pi\-dal irreducible 
subquotient~$\pi$.
There is a simple stratum $[\La_{{\rm max}},n_{{\rm max}},0,\b]$ 
in $\A=\Mat_m(\D)$ such that the restriction of $\pi$ to the 
pro-$p$-subgroup $\H^1_{{\rm max}}=\H^1(\b,\La_{{\rm max}})$ 
contains a simple character $\tmax\in\Cc(\La_{{\rm max}},0,\b)$.

\begin{lemm}
There is an irreducible subquotient $\tau$ of $\s$ such that 
$\tmax$ occurs in the restriction of $\Ind^\G_\P(\tau)$ to 
$\H^1_{{\rm max}}$. 
\end{lemm}

\begin{proof}
Since any representation of $\H^1_{{\rm max}}$ is semisimple, 
$\tmax$ is a direct summand of 
the res\-triction of $\Ind^\G_\P(\s)$ to $\H^1_{{\rm max}}$.
We fix an embedding~$\iota$ of $\tmax$ in $\Ind_\P^\G(\sigma)$ and write $\W$ for the (one-dimensional) image of $\tmax$ by $\iota$. 
Write $\V$ for the representation of finite type 
$\ind^\G_{\H^1_{{\rm max}}}(\tmax)$. 
If we write $\N$ for the uni\-potent radical of $\P$, 
Frobenius reciprocity gives us a nonzero homomorphism: 
\begin{equation*}
\iota_*:\V_\N\to\s.
\end{equation*} 
Write $\s_1$ for the image of this homomorphism.
It has the following properties: 
\begin{enumerate}
\item 
if $\s'$ is a proper subrepresentation of~$\s_1$ then $\Ind^\G_\P(\s')\cap \W=0$; 
\item
it is of finite type, since $\V$ 
is of finite type and Jacquet functors preserve finite type. 
\end{enumerate}
This implies that $\s_1$ has a maximal proper subrepresentation $\s_2$ 
and that the image of~$\V$ in the representation 
$\Ind^\G_\P(\s_1/\s_2)$ is non-zero. 
In particular $\tmax$ occurs in $\Ind^\G_\P(\s_1/\s_2)$ and $\s_1/\s_2$ is
an irre\-du\-cible subquotient of $\s$. 
\end{proof}

We may assume that $\M$ is a standard Levi subgroup, attached to 
a composition $(m_1,\dots,m_r)$ of $m$. 
Thus $\tau$ can be written on the form $\tau_1\odo\tau_r$, with $\tau_i$ an irreducible 
representation of $\G_{m_i}$, for each $i\in\{1,\dots,r\}$. 
Let $(\BJ_i,\bl_i)$ be a semisimple supertype 
of $\G_{m_i}$ occurring in $\tau_i$.
Then $\tmax$ occurs in: 
\begin{equation*}
\ind^{\G_{m_1}}_{\BJ_1}(\bl_1)\tdt\ind^{\G_{m_r}}_{\BJ_r}(\bl_r)
\simeq\ind^\G_{\BJ}(\bl)
\end{equation*}
where $(\BJ,\bl)$ is a suitable semisimple supertype of $\G$. 
We fix a decomposition $\bl=\bk\otimes\bs$ and thus get a functor $\KM$. 
As in the first part of the proof of Lemma~\ref{Cont},
it follows that $\KM(\pi)$ is nonzero.
By~\cite[Lemme 5.3]{MS1}, it is a finite direct sum of super\-cuspidal 
irreducible representations of $\Gg=\BJ/\BJ^1$.
By Theorem \ref{Cagliostro}, it is a subquotient of:
\begin{equation*}
\KM(\Ind^\G_\P(\s))\simeq\Ind^\Gg_{\Pp}(\KM_\M(\s)).
\end{equation*}
Thus Proposition~\ref{DDSCfini} gives us a contradiction. 
\end{proof}

\providecommand{\bysame}{\leavevmode ---\ }
\providecommand{\og}{``}
\providecommand{\fg}{''}
\providecommand{\smfandname}{\&}

%%%%%%%%%%%%%%%%%%%%%%%%%%%%%%%%%%%%%%%%%%%%%%%%%%%%%%%%%%%%%%%%%%%%%%%%

\end{document}